\documentclass[a4paper,12pt]{amsart}
\usepackage{amsfonts,amsthm,amssymb,amsmath,amscd}
\usepackage[colorlinks=true,linkcolor=blue,citecolor=blue]{hyperref}

\usepackage{tikz-cd}

\usepackage{tabularx}
\usepackage{tcolorbox}
\usepackage[english]{babel}
\usepackage{amssymb,amsmath}
\usepackage{xcolor}

\usepackage{tikz}
\usetikzlibrary{arrows}

\newtheorem{Theo}{Theorem}[section]
\newtheorem{Prop}[Theo]{Proposition}
\newtheorem{Coro}[Theo]{Corollary}
\newtheorem{Lemm}[Theo]{Lemma}

\newtheorem{Defi}[Theo]{Definition}

\newtheorem{Exam}[Theo]{Example}

\newtheorem{Rema}[Theo]{Remark}

\newcommand{\Hcal}{\mathcal{H}}
\newcommand{\T}{\mathbb{T}}
\newcommand{\Bcal}{\mathcal{B}}
\newcommand{\D}{\mathbb{D}}
\newcommand{\C}{\mathbb{C}}
\newcommand{\Z}{\mathbb{Z}}

\def\N{\mathbb{ N}}
\def\R{\mathbb{ R}}

\begin{document}

\title{$\pmb{\mathcal{H}_{p}}$-theory of general Dirichlet series}


\author[Defant]{Andreas Defant}
\address[]{Andreas Defant\newline  Institut f\"{u}r Mathematik,\newline Carl von Ossietzky Universit\"at,\newline
26111 Oldenburg, Germany.
}
\email{defant@mathematik.uni-oldenburg.de}

\author[Schoolmann]{Ingo Schoolmann}
\address[]{Ingo Schoolmann\newline  Institut f\"{u}r Mathematik,\newline Carl von Ossietzky Universit\"at,\newline
26111 Oldenburg, Germany.
}
\email{ingo.schoolmann@uni-oldenburg.de}

\maketitle

\begin{abstract}
\noindent
Inspired by results of Bayart on ordinary Dirichlet series $\sum a_n n^{-s}$, the main purpose of this article is to start an  $\mathcal{H}_p$-theory of general
Dirichlet series $\sum a_n e^{-\lambda_{n}s}$\,.
Whereas the $\mathcal{H}_p$-theory of ordinary Dirichlet series, in view of an ingenious identification of Bohr, may be seen as a sub-theory of Fourier analysis on  the infinite dimensional
torus $\mathbb{T}^\infty$, the $\mathcal{H}_p$-theory of general Dirichlet series is
build as a sub-theory of Fourier analysis on certain compact abelian groups, including the  Bohr compactification $\overline{\R}$ of the reals. Our approach allows to extend various important facts on Hardy spaces
of ordinary Dirichlet series to the much wider  setting of $\mathcal{H}_p$-spaces of general Dirichlet series.
\end{abstract}



%
%
\noindent
\renewcommand{\thefootnote}{\fnsymbol{footnote}}

\section{Introduction}


Within the last two decades the theory of  ordinary Dirichlet series $\sum a_{n} n^{-s}$ saw a remarkable renaissance which in particular led to the solution of some long-standing problems.

\medskip

 A fundamental object in these investigations
 is given by the Banach space $\Hcal_{\infty}$  of all ordinary Dirichlet series $D:=\sum a_{n} n^{-s}$ which converge and define a bounded, and then necessarily holomorphic, function on the open right half plane $[Re>0]$
 (endowed with the supremum norm on $[Re>0]$).
 \medskip

 One of the celebrated results in this 'ordinary' theory is a result of Hedenmalm, Lindqvist and Seip from  \cite{HLS} which shows that $\Hcal_{\infty}$ equals the Hardy space $H_{\infty}(\T^{\infty})$ on the infinitely dimensional torus as Banach spaces . Let us explain this in more detail.
 The infinitely dimensional torus $\T^{\infty}$ is the  infinite product of $\mathbb{T} = \{w \in \mathbb{C}\colon |w|=1\}$ which forms a natural compact abelian group
 on which the Haar measure is given by the  normalized Lebesgue measure. The characters of this group are the monomials $z \mapsto z^\alpha$, where $\alpha$ is a finite sequence
 of integers, and $H_{\infty}(\T^{\infty})$ denotes the closed subspace of all $f \in L_\infty(\mathbb{T}^\infty)$ such that the Fourier coefficient
 $
 \widehat{f}(\alpha) = \int_{\T^{\infty}} f(w) w^{-\alpha} dw =0\,,
 $
 whenever $\alpha < 0$ (in the sense that some $\alpha_k < 0$). Then there is a unique linear isometry
 $$\Bcal \colon H_{\infty}(\T^{\infty}) \to \Hcal_{\infty}\,,\,\,\,\, f\sim \sum_{\alpha \in \mathbb{N}^{(\mathbb{N})}} \widehat{f}(\alpha) z^\alpha  \mapsto D= \sum_{n\in \N} a_n n^{-s}\,,$$
  which preserves Fourier- and Dirichlet coefficients in the sense that
 \begin{equation} \label{vision}
 \widehat{f}(\alpha) = a_{n}\,\, \text{ whenever  } n= p^\alpha.
 \end{equation}
 The crucial point here is the fact that each natural number has a unique prime number decompostion
 (together with other nontrivial tools like e.g. Diophantine approximation), and the result shows that the theory of ordinary Dirichlet which generate bounded, holomorphic functions on the positive half plane is intimately linked with Fourier analysis on the group $\T^{\infty}$.

\medskip

 More generally, Bayart in \cite{Bayart} developed an $H_p$-theory of Dirichlet series. Recall that the Hardy space $H_p(\mathbb{T}^\infty)\,, \, 1 \leq p \leq \infty$, is the  closed subspace  of all $f \in L_p(\mathbb{T}^\infty)$
 such that $\widehat{f}(\alpha) =0$ if $\alpha<0$. Then the  Banach spaces $\mathcal{H}_p$ of ordinary Dirichlet series by definition (!) is the isometric image of $H_p(\mathbb{T}^\infty)$ under the identification
 from \eqref{vision}.

 \medskip

  But all these ideas fail  for general Dirichlet series.
  Given  a frequency $\lambda:=(\lambda_{n})$ (i.e.  a non-negative strictly increasing sequence of real numbers tending to $\infty$), a
$\lambda$-Dirichlet series is a formal sum
 $$\sum a_{n} e^{-\lambda_{n}s}$$
 with complex Dirichlet coefficients $a_{n}$ and a complex variable $s$. We point out that making the jump from the ordinary case $\lambda=(\log n)$ to arbitrary frequencies reveals challenging consequences.

  \medskip

 Much of the ordinary theory relies on 'Bohr's theorem', the fact that for each ordinary Dirichlet series the abscissas of uniform convergence and boundedness coincide. This phenomenon  fails  for general Dirichlet series. Further due to the fundamental theorem of arithmetics each natural number $n$ has its prime number decomposition $n=\mathfrak{p}^{\alpha}$ and so  the frequency $(\log n)$ can be written as a linear combination of $(\log p_{j})$ with natural coefficients. This intimately links the theory of ordinary Dirichlet series  with the theory of holomorphic functions on polydiscs, and in particular with the theory of polynomials $\sum c_{\alpha} z^{\alpha}$
in finitely many complex variables. One of several consequences is that  $m$-homogeneous Dirichlet series $\sum a_n n^{-s}$, i.e.  $a_n \neq 0$ only if $n$ has $m$ prim factors, are linked with $m$-homogeneous polynomials.
 This way powerful tools enter the game, as e.g. polynomial inequalities (like the Bohnenblust-Hille inequalities, hypercontractivity of convolution with the Poisson kernel, etc.), $m$-linear forms, or  polarization. So Fourier analysis on the infinite dimensional torus $\mathbb{T}^\infty$ and infinite dimensional holomorphy on the
 open unit ball $B_{c_0}$ of $c_0$ enrich the theory of ordinary Dirichlet series considerably. But unfortunately facing  general Dirichlet series   many of these powerful bridges
 seem to collapse. New questions arise which make the theory of general Dirichlet quite  demanding.

 \medskip
We define the space $\mathcal{D}_{\infty}(\lambda)$  analogously to the space $\mathcal{H}_{\infty}$.
 So $\sum a_{n} e^{-\lambda_{n}s}$ belongs to  $\mathcal{D}_{\infty}(\lambda)$ whenever it converges and defines a bounded, and then necessarily holomorphic, function on $[Re>0]$.
 For instance, if $\lambda=(n)$, then looking at the transformation $z=e^{-s}$ we  easily conclude that  $\mathcal{D}_{\infty}((n))$ is simply $H_{\infty}(\D)$, the space of all bounded and holomorphic functions on the open unit ball $\D$. And if $\lambda=(\log n)$, then we are in the  ordinary case. Hence, in both cases the space $\mathcal{D}_{\infty}(\lambda)$ can be described in terms of Fourier analysis, that is, it can be considered as a Hardy space, namely $\mathcal{D}_{\infty}((n))=H_{\infty}(\T)$ and $\mathcal{D}_{\infty}((\log n))=H_{\infty}(\T^{\infty})$. In view of these two examples the following  question arises naturally:
 \emph{Given an arbitrary frequency $\lambda$,
is it possible to describe $\mathcal{D}_{\infty}(\lambda)$ in terms of a sort of Hardy space on a compact abelian group?}

\medskip

Inspired  by ideas of Bohr and Helson we show that in this much more general situation,  the so-called Bohr compactification $\overline{\R}$ of the real numbers $\R$ is a suitable substitute for the infinite dimensional torus $\mathbb{T}^\infty$. We note that $\overline{\R}$ being a compact abelian group carries a Haar measure, contains   $\mathbb{R}$ as a dense subset, and that its dual group may be identified again with $\R$
(see Section~\ref{Bohrcompact} for the precise definitions). So in particular ordinary Dirichlet series can be described in terms of functions on $\overline{\R}$. This might be surprising since in contrast to the infinitely dimensional torus $\T^{\infty}$ the Bohr compactification $\overline{\R}$ is not metrizable and is moreover a sort of one-dimensional object.

\medskip

Motivated by this observation we introduce a Fourier analysis setting for the study of  general Dirichlet series.
We restrict ourserlves to $\lambda$-Dirichlet series with Dirichlet coefficients that  actually are Fourier coefficients defined by functions on compact abelian groups $G$ of a certain type, namely  compact abelian groups allowing a  continuous homomorphism $\beta \colon \R \to G$ with dense range (this includes $\T$ and $\T^{\infty}$).
  Given a frequency $\lambda=(\lambda_n)$, we call such a pair $(G,\beta)$ a $\lambda$-Dirichlet group whenever  every character $e^{-i\lambda_n \cdot}: \mathbb{R}\to \mathbb{T}$
 has an  'extension' $h_{\lambda_n}$ (which then is unique) as a character on $G$:
\begin{equation*}
\begin{tikzpicture}[scale = 0.8]
        \node (G) at (0,0) {$G$};
        \node (T) at (3,0) {$\mathbb{T}$};
        \node (R) at (0,-2) {$\mathbb{R}$};

        \draw[-latex] (G) -- node[above] {$h_{\lambda_n}$} (T);
        \draw[-latex] (R) -- node[below right] {$e^{-i\lambda_n \cdot}$} (T);
        \draw[-latex] (R) -- node[left] {$\beta$} (G);
    \end{tikzpicture}
    \end{equation*}
To see a first non-trivial example, look at the continuous homomorphism
 \[
 \beta_{\T^\infty}: \R  \rightarrow \mathbb{T}^\infty\,, \,\, t \mapsto (p_k^{-it})\,.
 \]
 Then  by Kronecker's approximation theorem the pair
  $(\mathbb{T}^\infty, \beta_{\T^\infty})$ forms a $\lambda$-Dirichlet group of the frequency $\lambda = (\log n)$.

  \medskip

  Now given such a $\lambda$-Dirichlet group $(G,\beta)$ and $1 \leq p \leq \infty$, we define the Hardy space
  $H_{p}^{\lambda}(G)$ of all $f \in L_{p}(G)$ having  Fourier transforms $f: \widehat{G} \to \mathbb{C}$ supported in $\{h_{\lambda_n} \colon n \in \mathbb{N}\}$.
  Then the Hardy space $\mathcal{H}_{p}(\lambda)$ consists of all Dirichlet series
  $$ D = \sum \widehat{f}(h_{\lambda_{n}})e^{-\lambda_{n}s}\,,\,\,\, f \in H_p^\lambda(G)$$
which  together with the norm $\|D\|_{p}:=\|f\|_{p}$
   forms a Banach space.

   \medskip
 In fact it will turn out that $\mathcal{H}_{p}(\lambda)$ is independent of the chosen $\lambda$-Dirichlet group $(G,\beta)$.
 Looking at   $\lambda = (\log n)$ and the $\lambda$-Dirichlet group  $(\mathbb{T}^\infty, \beta_{\T^\infty})$,
Bayart's theory of Hardy space $\mathcal{H}_p$ of Dirichlet series is contained in the $\mathcal{H}_p$-theory of general Dirichlet series we intend to present.

   \medskip

Assuming some restrictions  on the frequency $\lambda$ and given a $\lambda$-Dirichlet group $(G,\beta)$, we prove  in  Theorem \ref{mainresultinfty} that $H_\infty^\lambda(G)$ and  $\mathcal{D}_{\infty}(\lambda)$
by means of the identification $f \mapsto \sum \widehat{f}(h_{\lambda_n}) e^{-\lambda_n s}$  coincide isometrically, and hence in particular
\begin{equation} \label{sane}
\mathcal{D}_{\infty}(\lambda) = \mathcal{H}_{\infty}(\lambda)\,.
\end{equation}
The crucial idea for the proof is inspired by Helson's  theorem~\ref{helson}  from \cite{Helson3}
which shows that almost all vertical limits of a general Dirichlet series $\sum a_n e^{-\lambda_{n}s}$
with square summable coefficients, converge pointwise on the positive half plane. In \cite{DefantSchoolmann3} we extend this result to $\mathcal{H}_{p}(\lambda)$, $1\le p <\infty$, adding the relevant maximal inequality.
We point out that in \cite{HLS} the equality $\mathcal{H}_{\infty}=H_{\infty}(\T^{\infty})$ is derived by first showing that $\Hcal_{\infty}$ equals the space $H_{\infty}(B_{c_{0}})$ of all bounded and holomorphic functions on the unit ball of $c_{0}$ and then proving $H_{\infty}(B_{c_{0}})=H_{\infty}(\T^{\infty})$, which can be seen as an 'infinite dimensional variant' of the classical fact $H_{\infty}(\T)=H_{\infty}(\D)$. Here we give a somewhat  direct proof of $\Hcal_{\infty}=H_{\infty}(\T^{\infty})$ using tools from Fourier analysis on $\R$.

\medskip

Inspired by the isometric equality from~\eqref{sane}, we start
 a sort of systematic structure theory of
$\mathcal{H}_p(\lambda)$-spaces  of general Dirichlet series  modelled along Bayart's $\mathcal{H}_p$-theory of ordinary Dirichlet series.
In Section~\ref{fouriersetting} we invent our setting.
In Section \ref{section4} we introduce   frequencies $\lambda$ of  integer and natural type, and show how to reproduce
important facts on ordinary Dirichlet series (due to Bohr, Bohnenblust-Hille, Hedenmalm-Lindqvist-Seip, and Bayart) from their  analogs on general Dirichlet series (see e.g. Theorem \ref{integertype}  and Corollary~\ref{ordinary}).
In Section \ref{structuretheory} we extend some key  results on ordinary Dirichlet series  to our more general setting. In Theorem~\ref{schauderbasisinHp} we show that
the monomials $e^{-\lambda_{n} s}$ form  a Schauder basis of  $\Hcal_{p}(\lambda),\, 1 < p < \infty$.
For $p=1$ we give upper estimates of the basis constant of the
$e^{-\lambda_{n} s}$ in $\Hcal_{1}(\lambda)$ (Corollary~\ref{basisconstant}). In Theorem~\ref{montel} we generalize Bayart's  Montel type theorem from $\mathcal{H}_\infty= \mathcal{D}_\infty(\log n)$ to  $\mathcal{H}_p(\lambda)$, and in Theorem~\ref{Nabschnitt}
we show how Dirichlet series in $\mathcal{H}_p(\lambda)$ are determined by their
$N$th Abschnitte.  And finally we in  Theorem~\ref{brotherRiesz} finish with a  `brothers Riesz theorem' identifying  $\mathcal{H}_1(\lambda)$
with certain analytic measures on $\lambda$-Dirichlet groups.

\section{General Dirichlet series}
As already mentioned in the introduction a general Dirichlet series is a formal sum of the form $\sum a_{n} e^{-\lambda_{n}s}$, where $(\lambda_{n})$ is a strictly increasing non negative sequence of real numbers (called frequency), $(a_{n})$ a sequence of complex coefficients (called Dirichlet coefficients), and $s$ a complex variable. All (formal) $\lambda$-Dirichlet series are denoted by $\mathcal{D}(\lambda)$.
 The following `abscissas' rule the convergence theory of general  Dirichlet series $D=\sum a_{n}e^{-\lambda_{n}s}$:
\begin{align*}
&
\sigma_{c}(D)=\inf\left \{ \sigma \in \R \mid D \text{ converges on } [Re>\sigma] \right\},
\\&
\sigma_{a}(D)=\inf\left \{ \sigma \in \R \mid D \text{ converges absolutely on } [Re>\sigma] \right\},
\\&
\sigma_{u}(D)=\inf\left \{ \sigma \in \R \mid D \text{ converges uniformly on } [Re>\sigma] \right\}.
\end{align*}
By definition $\sigma_{c}(D)\le \sigma_{u}(D)\le \sigma_{a}(D)$ and in  general all these abscissas differ.
Let us mention that general Dirichlet series always define holomorphic functions on $[Re>\sigma_{c}(D)]$
(see e.g. \cite{HardyRiesz}).

\medskip

In \cite{DefantSchoolmann2} we introduce the following two spaces of $\lambda$-Dirichlet series. The space $\mathcal{D}^{ext}_{\infty}(\lambda)$ of all somewhere convergent $D \in \mathcal{D}(\lambda)$ allowing a holomorphic and bounded extension $f$ to the open right half plane $[Re>0]$ endowed with the norm $\|D\|_{\infty}:=\|f\|_{\infty}=\sup_{[Re>0]}|f(s)|$ (this in fact is a norm, see \cite[Corollary 3.8]{DefantSchoolmann2} or Corollary~\ref{H2}).
Moreover, $\mathcal{D}_{\infty}(\lambda)$ is the space of all  $\lambda$-Dirichlet series which converge and define  a bounded (and then necessarily holomorphic) function on $[Re>0]$. By definition we have $\mathcal{D}_{\infty}(\lambda) \subset \mathcal{D}^{ext}_{\infty}(\lambda)$, and in \cite[Theorem 5.2]{DefantSchoolmann2} it is shown that in general this inclusion is strict and $\mathcal{D}_{\infty}(\lambda)$ fails to be complete.

\medskip

 But there are sufficient conditions on $\lambda$ that forces $\mathcal{D}_{\infty}(\lambda) = \mathcal{D}^{ext}_{\infty}(\lambda)$  and the completeness of $\mathcal{D}_{\infty}(\lambda)$. Let us collect a few 'analytic' conditions.
   The first one was  isolated by Bohr (see \cite{Bohr}), we denote it by $(BC)$:
\begin{equation*} \label{BC}
 \exists ~l = l (\lambda) >0 ~ \forall ~\delta >0 ~\exists ~C>0~\forall ~n \in \N: ~~\lambda_{n+1}-\lambda_{n}\ge Ce^{-(l+\delta)\lambda_{n}};
\end{equation*}
roughly speaking this condition prevents the $\lambda_n$'s from getting too close too fast. Secondly a strictly weaker condition of Landau $(LC)$ (see \cite{Landau}):
$$\forall ~\delta>0 ~\exists ~C>0~ \forall ~n \in \N \colon ~~ \lambda_{n+1}-\lambda_{n}\ge C e^{-e^{\delta\lambda_{n}}}.$$
Another important, say geometric, value associated to $\lambda$ is the maximal width of the strip of convergence and non absolutely convergence:
$$L(\lambda):=\sup_{D \in \mathcal{D}(\lambda)} \sigma_{a}(D)-\sigma_{c}(D)=\limsup_{N\to \infty} \frac{\log(n)}{\lambda_{n}}=\sigma_{c}(\sum e^{-\lambda_{n}s}).$$
Much of the abstract theory of ordinary Dirichlet series is based on a fundamental theorem of Bohr \cite{BohrStrip} which shows that  every $D \in \mathcal{D}^{ext}_\infty((\log n))$ converges uniformly on all half spaces $[\text{Re} > \varepsilon]$ for  all
$\varepsilon >0$, and this then easily implies that $\mathcal{D}_\infty^{\text{ext}}((\log n)) = \mathcal{D}_\infty((\log n))$. In \cite{Bohr} and \cite{Landau} Bohr and Landau extended this result to general Dirichlet series for frequencies
satisfying $(BC)$  or $(LC)$, respectively.

\begin{Defi}
We say that a frequency $\lambda$ satisfies Bohr's theorem, whenever every $D \in \mathcal{D}^{ext}_\infty(\lambda)$ converges uniformly on all half spaces $[\text{Re} > \varepsilon]$ for  all
$\varepsilon >0$.
\end{Defi}

As in the ordinary case the identity principle implies that for every $\lambda$ with Bohr's theorem we have $$\mathcal{D}_\infty^{ext}(\lambda) = \mathcal{D}_\infty(\lambda)\,.$$
Some of the forthcoming results of this article hold under asumptions like `$\lambda$
 satisfies Bohr's theorem' and/or `$\mathcal{D}_\infty^{\text{ext}}(\lambda) = \mathcal{D}_\infty(\lambda)$' and/or  `$ \mathcal{D}_\infty(\lambda)$ is complete'.
 Analyzing  the works of Bohr and Landau, we in  \cite[Remark 4.8 and Theorem 5.1]{DefantSchoolmann2} collect a few conditions which in concrete cases allow
 to verify  these conditions for concrete frequencies.

\begin{Theo} \label{conditions2}
A frequency  $\lambda$ satisfies Bohr's theorem   if one of the following conditions holds:
\begin{enumerate}
\item[(1)] $L(\lambda)=0$,
\item[(2)] $\lambda$ is $\mathbb{Q}$-linearly independent,
\item[(3)] $(LC)$ (this includes $(BC)$).
\end{enumerate}
Moreover, in each of these cases we have  $\mathcal{D}_\infty^{\text{ext}}(\lambda) = \mathcal{D}_\infty(\lambda)$, but non of these conditions is  necessary for Bohr's theorem.
\end{Theo}
Let us face completeness.
\begin{Theo} \label{conditions}
Let $\lambda$ be a frequency. Then $\mathcal{D}_{\infty}(\lambda)$ is complete if one of the following conditions holds
\begin{enumerate}
\item[(1)] $L(\lambda)=0$,
\item[(2)] $\lambda$ is $\mathbb{Q}$-linearly independent,
\item[(3)] $L(\lambda)<\infty$ and  $\mathcal{D}^{ext}_{\infty}(\lambda)=\mathcal{D}_{\infty}(\lambda)$.
\end{enumerate}
 Moreover, non of these conditions is necessary.
\end{Theo}

 For later use we  finally
  recall that Dirichlet series $D \in \mathcal{D}^{ext}_{\infty}(\lambda)$ are approximated by their so-called typical means (see again \cite[\S 3]{DefantSchoolmann2}).
\begin{Prop} \label{typicalmeans}
Let $D=\sum a_{n} e^{-\lambda_{n}s} \in \mathcal{D}^{ext}_{\infty}(\lambda)$ with extension $f$. Then
the sequence of polynomials
$$R_{x}(D)(s):=\sum_{\lambda_{n}<x} a_{n} \left(1-\frac{\lambda_{n}}{x}\right)
e^{-\lambda_{n}s}\,,\,\, x>0$$
\noindent
converges uniformly to $f$ on $[Re>\varepsilon]$ as $x\to \infty$ for all $\varepsilon>0$.
\end{Prop}
\section{Fourier analysis setting} \label{fouriersetting}

In this section we  present a new abstract Fourier analysis approach to general Dirichlet series.
 Basically
we restrict ourselves to general Dirichlet series with Dirichlet coefficients which actually are Fourier coefficients of  functions on  certain compact
abelian groups.
This has several  advantages.
One is  that  the class of all general Dirichlet series simply is  too large to obtain a good understanding.
 Assuming that the Dirchlet coefficients are Fourier coefficients gives more structure and allows to use tools from harmonic analysis like  the Hausdorff-Young inequality or Plancherel's theorem (among others).
   A further advantage of our setting is that  Bayart's  $\mathcal{H}_{p}$-theory of ordinary Dirichlet series embeds in a natural way. Whereas the $\mathcal{H}_{p}$-theory of ordinary Dirichlet series is basically Fourier
  analysis on the infinite dimensional torus  $\T^{\infty}$, this group  fails to be the right model
  for general Dirichlet series. In fact, the Bohr compactification $\overline{\R}$ of $\R$ and products of $\widehat{\mathbb{Q}_d}$ (the dual group of the rationals endowed with the discrete topology) turn out to be  suitable substitutes. Finally, fixing some $\lambda$, regarding the different realisations of $\lambda$-Dirichlet series of this type, another feature of our approach is that the $\mathcal{H}_{p}$-theory of general Dirichlet series we intend to present will be independent of the chosen suitable group for $\lambda$.

\subsection{Dirichlet groups}
We start by introducing the type of groups we are interested in.

\begin{Defi} Let $G$ be a compact abelian group and $\beta\colon \R  \to G$ a homomorphism. Then we call the pair $(G, \beta)$ a Dirichlet group if $\beta$ is continuous and has dense range.
\end{Defi}
\noindent

We will frequently use the notion of dual maps of homomorphisms $T \colon G \to H$, where $G$ and $H$ are locally compact abelian groups. The dual map $\widehat{T} \colon \widehat{H} \to \widehat{G}$ is then defined by $\widehat{T}(h):= h \circ T$ (see \cite[\S 1.2.1, p.6]{Rudin62} for the definition of  the dual group $\widehat{G}$). By Pontryagin's theorem (see \cite[\S 1.7.2, Theorem, p. 28]{Rudin62}) we have $\widehat{\widehat{T}}=T$, and $\widehat{T}$ is continuous if $T$ is continuous. Moreover, we need the following fact from the theory of locally compact groups.

\begin{Lemm}\label{dualmapping} Let $T \colon G \to H$ be a continuous homomorphism of locally compact abelian groups. Then $T$ is injective if and only if $\widehat{T}$ has dense range.
\end{Lemm}
For the sake of completeness we give a proof. Given a locally compact  abelian group $G$  with unit $e$ and  a subgroup $U$ of $G$,  we call $U^{\perp}:=\{ \gamma \in \widehat{G} \mid  \gamma(u)=e \text{ for all } u \in U \}$ the orthogonal group of $U$. Then ${(U^{\perp})}^{\perp}=\overline{U}$, where the closure is taken in $G$ (see  \cite[\S1.5.2]{QQ}).
\begin{proof}[Proof of Lemma \ref{dualmapping}]
By definition $\ker \widehat{T}=(\text{Im} T)^{\perp}$, and consequently $(\ker \widehat{T})^{\perp}=\overline{\text{Im} T}$. Using $\widehat{\widehat{T}}=T$ gives $\ker T=(\text{Im} \widehat{T})^{\perp}$, and hence $(\ker T)^{\perp}=\overline{\text{Im} \widehat{T}}$ which proves the claim.
\end{proof}
\noindent
With this, given a Dirichlet group $(G, \beta)$, the dual map of $\beta$ given by
\begin{equation*}
\widehat{\beta} \colon \widehat{G} \to \widehat{\R}, ~~\gamma \mapsto \gamma \circ \beta
\end{equation*}
is a monomorphism.
Using $\widehat{\R}=\R$ (as topological groups) implies that for  every $\gamma$ there is a unique $x \in \R$ such that $\gamma \circ \beta=e^{-ix\cdot}$. Identifying $\widehat{\beta}(\gamma)=x$ we obtain an embedding
$$\widehat{\beta} \colon \widehat{G} \hookrightarrow \R.$$

\medskip

Another property of Dirichlet groups we are going to use is  that they are connected (see the proof of Theorem~\ref{schauderbasisinHp}); this is an immediate consequence of the fact that, given a connected set,
its   image under a continuous map  as well as its closure are again connected.

\medskip

Observe that every character $e^{-ix\cdot} \in \widehat{\beta}(\widehat{G})$, allows a unique `extension' to $G$, say $h_{x}$,  such that $h_{x} \circ \beta=e^{-ix\cdot}$.
Conversely, for every $\gamma \in \widehat{G}$ there is some $x \in \mathbb{R}$ such that $e^{-ix\cdot} = \gamma\circ \beta$,  which proves the following
\begin{Prop}\label{hx} Let $(G, \beta)$ be a Dirichlet group.
The characters $e^{-ix\cdot} \colon \R \to \T$ on $\R$, where $x\in \widehat{\beta}(\widehat{G})$, are precisely those which allow a continuous `extension' $h_{x}$ to $G$. In particular
\begin{equation*} \label{dualmapping3}
\widehat{G}=\{h_{x} \mid x \in \widehat{\beta}(\widehat{G}) \}.
\end{equation*}
\end{Prop}
More precisely, we should denote the characters of $G$ by $h_{x}^{(G,\beta)}$ instead of $h_{x}$, but in most situations this in fact will not be necessary.\\

Motivated by the equality $G=\widehat{\widehat{G}}$ (Pontryagin's duality theorem) for notational convenience we for all $\omega \in G$  write
$$\omega(x):=h_{x}(\omega), ~ x \in \widehat{\beta}(\widehat{G}).$$

Let us consider some examples of Dirichlet groups.
As explained, if $(G, \beta)$ is a Dirichlet group, then $\widehat{G} \subset \R$ (via $\widehat{\beta}$).
In fact,  the reverse of this implication holds, which gives a handy criterion to establish a list of main examples of  Dirichlet groups. The proof of the following proposition  is an  immediate consequence of Lemma \ref{dualmapping}; recall that the dual group $\widehat{G}$  carries the discrete topology, since $G$ is compact (see \cite[\S 1.2.5, Theorem, p. 9]{Rudin62}).

\begin{Prop} \label{aya} Let $G$ be a compact abelian group. If there is an injective homomorphism $T \colon \widehat{G}\hookrightarrow \R$, then $(G, \widehat{T})$ is a Dirichlet group.
\end{Prop}
We write $G_d$ for a group $G$ with the discrete topology $d$. In the following we give four central examples --
and start with the `mother' of all of them.

\begin{Exam}\label{examples}
 Let $U$  be a subgroup of $\R$. Then the topological group  $\widehat{U_d}$ together with the
 mapping
$$ \beta_{\widehat{U_d}}\colon \R \to  \widehat{U_d}, ~~ t \mapsto \left[u \mapsto e^{-itu}\right] $$
 forms a  Dirichlet group. In particular, for $U = \mathbb{Z}$ and identifying $\T =\widehat{\mathbb{Z}_d}$, we
 obtain the Dirichlet group $(\T, \beta_{\T})$, where
\[
\beta_{\T}: \R \rightarrow \T, \, t \mapsto e^{-it}.
\]
\end{Exam}

\begin{Exam}\label{examples1}
The compact abelian group  $\overline{\R}:=\widehat{(\R,d)}$ is  the so-called   Bohr compactification
of $\R$ which forms a Dirichlet group with the embedding
\begin{equation*}
\beta_{\overline{\R}}\colon \R \hookrightarrow \overline{\R},\,\, x \mapsto \left[ t \mapsto e^{-ixt}\right]\,.
\end{equation*}
\end{Exam}

In Proposition \ref{biggestDirichletgroup} we will see that $(\overline{\R},\beta_{\overline{\R}})$ is in some sense the smallest Dirichlet group.

\begin{Exam}\label{examples2}
Let $B:=(b_{1}, b_{2}, \ldots)$ be a $\Z$-linearly independent (equivalently, $\mathbb{Q}$-linearly independent) sequence of real numbers of length $N \in \mathbb{N} \cup \{\infty\}$ (choose for instance $b_{n}:=\log(p_{n})$ where $p_{n}$ is the $n$th prime number). Then $\bigoplus_{n=1}^{N} \Z$
in view of the injective homomorphism
$$\bigoplus_{n=1}^{N} \Z \hookrightarrow \R, ~~\alpha \mapsto \sum \alpha_{j}b_{j}$$
is a subgroup of $\R$, and hence the compact abelian group
$$\T^{N} = \prod_{n=1}^{N} \widehat{\Z_d} = \widehat{\bigoplus_{n=1}^{N} \Z_d}$$
together with
$$\beta_{\T^{N}} \colon \R \to \T^{N},~~ t \mapsto  (e^{-it b_{j}})_{j=1}^N$$
forms a Dirichlet group.
\end{Exam}

The density of $\beta_{\T^{N}}$ is nothing else than Kronecker's theorem, which is a key argument in the theory of ordinary Dirichlet series (choosing the $\mathbb{Q}$-linearly independent sequence $B:=(\log p_{n})$, where $p_{n}$ is the $n$th prime number).

\begin{Exam}\label{examples3} \text{}
Let $B:=(b_{1}, b_{2}, \ldots)$ be a $\mathbb{Q}$-linearly independent sequence of real numbers of length $N \in \mathbb{N} \cup \{\infty\}$.
Then
\begin{equation} \label{themap}
T_B: \bigoplus_{n=1}^{N} \mathbb{Q} \hookrightarrow \R, ~~\alpha \mapsto \sum \alpha_{j}b_{j}
\end{equation}
is an  injective homomorphism, and its dual map
$$\widehat{T_{B}} \colon \R \to \prod_{n=1}^{N} \widehat{\mathbb{Q}_d}, ~~ t \mapsto \left[ (q_{j})_{j} \mapsto e^{-it\sum q_{j}b_{j}}\right]$$
 has dense range. In short
 $$\left(\widehat{\mathbb{Q}_d}^{N}, \widehat{T_{B}}\right):=\left(\prod_{n=1}^{N} \widehat{\mathbb{Q}_d},\widehat{T_{B}}\right)$$ is a Dirichlet group.
\end{Exam}
 To see the relevance of these examples within the study of general Dirichlet series we refer to  Section \ref{severalvariables}.
In view of the preceding examples we explain in which sense the Bohr compactification $\overline{\R}$ is the 'smallest' Dirichlet group. Later this in particular allows to link $L_{p}$-spaces on different Dirichlet groups.
\begin{Prop}\label{biggestDirichletgroup} For every Dirichlet group $(G, \beta)$ there is a unique
continuous homomorphism
\[
\pi_G: \overline{\R} \rightarrow G
\]
with dense range and such that
 $\beta = \pi_G \circ \beta_{\overline{\R} } $. Moreover, $\overline{\R}$ is the unique Dirichlet group with this property up to isomorphisms of groups.
\end{Prop}

\begin{proof}
We define $\pi_{G}$ to be the dual map of  $\iota\colon \widehat{G} \to (\R,d),~ \iota(\gamma):=\widehat{\beta}(\gamma)$
(identifying $\widehat{(\R, |\cdot|)} = (\R,d)$). Indeed, $\iota$ is injective, and consequently
$\pi_{G}:=\widehat{\iota}\colon \overline{\R} \to G$ has dense range by Lemma \ref{dualmapping} (and Pontryagin's duality theorem).
To  see the second statement, assume that $(H, \beta_H)$ is another Dirichlet group with the desired property. Then there are  two continuous
homomorphisms  $\pi_{\overline{\R}}\colon H \to \overline{\R}$ and $\pi_{H}\colon \overline{\R}\to H$ with dense range and such that $\beta_{\overline{\R}}=\pi_{\overline{\R}} \circ \beta_{H}$ and $\beta_{H}=\pi_{H} \circ \beta_{\overline{\R}}$. This implies
$\beta_{\overline{\R}}=\pi_{\overline{\R}}\circ \pi_{H} \circ \beta_{\overline{\R}}$ and $\beta_{H}=\pi_{H}\circ \pi_{\overline{\R}} \circ \beta_{H}$.
Since  $\beta_{H}$ and $\beta_{\overline{\R}}$  both have dense range, we get $ \pi_{\overline{\R}}\circ \pi_{H} =id_{\overline{\R}}$ and $\pi_{H}\circ \pi_{\overline{\R}}=id_{H}$, hence $H=\overline{\R}$.
\end{proof}

We close this section  introducing  several devices which transport tools from Fourier analysis on $\R$ to Fourier analysis on Dirichlet groups.
The first two  show how to  'restrict' functions on $G$ to $\R$, and we here  distinguish between  continuous functions  and integrable functions.

\begin{Prop} \label{basic}
Let $(G, \beta)$ be a Dirichlet group, and $f \in C(G)$. Then
\[
\int_G f(\omega) d\omega = \lim_{T \to \infty} \frac{1}{2T} \int_{-T}^T f(\beta(t)) dt\,.
\]
In particular, for every  polynomial
$\sum_{n=1}^N a_n h_{x_n}$ with $x_1, \ldots, x_N \in \hat{\beta}(\widehat{G})$ and $1\le p<\infty$
\begin{equation*}
 \int_{G} \bigg|\sum^{N}_{n=1} a_{n} h_{x_n}(\omega)\bigg|^{p} dm(\omega)=
\lim_{T \to \infty} \frac{1}{2T} \int_{-T}^{T}\bigg|\sum^{N}_{n=1} a_{n}e^{-ix_nt} \bigg|^{p} dt.
\end{equation*}
\end{Prop}
\begin{proof}
Obviously, the second equality in an immediate consequence of the first one.
On the left side of the first  equality we have  a bounded functional on $C(G)$, and the density of the range of $\beta$ implies that  the set of all functionals $f \to  \frac{1}{2T} \int_{-T}^T f(\beta(t)) dt, \,T >0$ is uniformly bounded on
$C(G)$. Then, by the density of the polynomials in $C(G)$ (see \cite[\S 8.7.3]{Rudin62}), it  suffices to check the desired equality  for all characters $h_x, x \in \widehat{\beta}(\widehat{G})$. For $x = 0$ the conclusion is trivial, and for $x \neq 0$ this follows from the fact that $\lim_{T \to \infty} \frac{1}{2T} \int_{-T}^T h_x(\beta(t)) dt=\lim_{T \to \infty} \frac{1}{2T} \int_{-T}^T e^{-ixt} dt =0$\,.
\end{proof}

Note that  Proposition \ref{basic} is a reformulation of an important fact of the theory of almost periodic functions.
Recall that  a continuous and bounded function $f \colon \R \to \C$ is almost periodic if and only if it has a continuous
  extension $F$ to $\overline{\R}$ such that $f= F \circ \beta_{\overline{\R}}$ (see \cite[\S 1.5, Theorem 1.5.5 and Theorem 1.5.6]{QQ}), and in this case
\begin{equation*}
\lim_{T\to\infty} \frac{1}{2T} \int_{-T}^{T} f(t) dt= \int_{\overline{\R}} F(\omega) dm(\omega).
\end{equation*}

\medskip

To deduce a similar result for integrable functions one should
have in mind  that $\beta(\R)$ may be a null set in $G$, hence  the restriction of
some function  $f \in L_{1}(G)$ to the subgroup $\beta(\R)$ of $G$ in fact may equal  $0$. Indeed, choosing the Dirichlet group $G=\T$ with $\beta(t):=e^{-it}$ we have $\beta(\R)=\T$, which is obviously not a null set in $\T$.
But, if $G=(\overline{\R}, \beta_{\overline{\R}})$ (Example \ref{examples1}), then $\beta_{\overline{\R}}(\mathbb{R})$ is a null set in $\overline{\R}$ (see \cite[Theorem 4.3]{Hewitt} or \cite{Varopoulos}).

\begin{Lemm} \label{restrictiontoreals}
Let $(G, \beta)$ be a Dirichlet group and $f \colon G \to \C$ be measurable. Then $f_{\omega}:=f(\omega \beta(\cdot)) \colon \R \to \C$ is measurable for almost all $\omega \in G$. If, additionally,
$f \in L_{\infty}(G)$, then
  $f_{\omega} \in L_{\infty}(\R)$ with $\|f_{\omega}\|_{\infty}\le \|f\|_{\infty}$ for almost all $\omega \in G$.
 Moreover, for every  $f \in L_{1}(G)$ and  $g \in L_{1}(\R)$ the convolution
$$g*f_{\omega}(t):=\int_{\R} f(\omega\beta(t-y)) g(y) dy$$
is almost everywhere defined on $\R$ and measurable.
\end{Lemm}

\begin{proof}
Denote by $\mu$ the product measure of the Lebesgue measure  $\lambda$ and the Haar measure $m$ on $\R\times G$.
Since the map
 \begin{equation} \label{regen}
 F\colon \R \times G \to \C,~~ (t,\omega) \mapsto f(\omega\beta(t))
 \end{equation}
  is measurable, $f_{\omega}$ is measurable on $\R$ for almost every $\omega \in G$ by the Fubini-Tonelli theorem.

For the proof of the second statement assume that  $f \in L_{\infty}(G)$.
   We claim  that $F \in L^{\infty}(\R \times G)$ and $\|F\|_{\infty} \le \|f\|_{\infty}$, which clearly  finishes the argument. Let $C>0$ and $N \subset G$ be a null set such that $|f(\omega)|\le C$ for all $\omega \in G\setminus N$. Since the map
$$M \colon \R  \times G \to G,\,\, M(t,\omega):=\beta(t)\omega$$ is continuous, the set $\widetilde{N}:=M^{-1}(N)$ is measurable.  For any fixed $y\in \R$ the set $\widetilde{N}_{y}:=N\beta(y)^{-1}=\{\omega \in G \mid \omega\beta(y) \in N\}$ is measurable. Since $m$ is translation invariant, we still have $m(\widetilde{N}_{y})=0$ for all $y \in \R$. Therefore  $\mu(\widetilde{N})=\int_{\R} m(\widetilde{N}_{y}) d\lambda(y)=0$. By definition, $|F(y,\omega)|\le C$ for all $(y,\omega) \notin \widetilde{N}$. Hence  $F \in L_{\infty}(\R\times G)$ and $\|F\|_{\infty}\le \|f\|_{\infty}$.\\
It remains to show the third statement, i.e. there is a null set $N \subset G$ such that for all $\omega \notin N$ the function $g*f_{\omega}$ is defined almost everywhere on $\R$ and measurable. Take the Poisson kernel $P_{1}(t):=\frac{1}{\pi}\frac{1}{1+t^{2}}$ on $\R$. Then  the map
$$F(\omega,t,y):=g(y)f(\omega\beta(t-y))P_{1}(t) \colon G \times \R \times \R \to \C$$
is measurable, and since the Haar measure $m$ on $G$ is translation invariant (and $\|P_{1}\|_{L_{1}(\R)}=1$) we obtain
\begin{align*}
\int_{\R}\int_{\R} \int_{G} |F(\omega,t,y)| dm(\omega) dt dy & \le \|g\|_{L_{1}(\R)} \|f\|_{L_{1}(G)}.
\end{align*}
So by the Fubini-Tonelli theorem there is a null set $N \subset G$ such that the map
$$\omega \mapsto \int_{\R}\int_{\R}g(y)f(\omega \beta(t-y))P_{1}(t) dy dt $$
is defined on $G\setminus N$ and measurable on $G$. Let $\omega \in G\setminus N$. Since
$$\int_{\R}\int_{\R}g(y)f(\omega\beta(t-y))P_{1}(t)dy dt < \infty,$$
another application of the Fubini-Tonelli theorem implies that the map
$$ t \mapsto \int_{\R}g(y)f(\omega\beta(t-y))P_{1}(t) dy= (g*f_{\omega})(t)P_{1}(t)$$
is defined for almost all $t \in \R$ and measurable. So  multiplying with the measurable function $t \mapsto P_{1}^{-1}(t)$ we obtain that
\begin{equation*}
g*f_{\omega}(t)=\int_{\R} g(y) f_{\omega}(t-y) dy
\end{equation*}
is defined almost everywhere on $\R$ and measurable.
\end{proof}
The second device allows to interpret functions in $L_{1}(\R)$ as measures on the Dirichlet group $G$. Therefore we denote by $M(G)$, $G$ any locally compact group, the Banach space of all bounded and regular complex Borel measures on $G$.
\begin{Lemm}\label{measuresembedd} Let $(G, \beta)$ be a Dirichlet group. Then there is a contraction $$\Phi \colon M(\R) \to M(G), ~~ \varphi \mapsto \mu$$ such that $\hat{\varphi} \circ \widehat{\beta}=\hat{\mu}$. In particular, every $f \in L^{1}(\R)$ defines a measure $\Phi(fd\lambda)$ on $G$. Moreover,   $\|f\|_{1}=\|\Phi(f)\|$
provided $f\ge 0$.
\end{Lemm}

\begin{proof}
Let $\varphi \in M(\R)$.  Then we define $\mu := \Phi(\varphi)$ to be the push forward measure of $\varphi$ with respect to $\beta$. Indeed, the functional
$$T_{\varphi}\colon C(G) \to \C, ~~ g \mapsto \int_{\R} g \circ \beta ~d\varphi$$
has norm  $ \le \|\varphi\|$, which by the Riesz's representation theorem implies that  $\mu \in M(G)$
with $ \|\mu\|\le \|\varphi\|$. Moreover,
$\widehat{\varphi}(x)=T_{\varphi}(\overline{h_{x}})=\widehat{\mu}(h_x)$ for all $x \in \widehat{\beta}(\widehat{G})$. For the last statement of the lemma note that $h_{0}\circ \beta=1$ so that $\|\Phi(f)\|\ge \left| \int_{\R} f(t) dt\right|=\|f\|_{1}$ provided $f \in L_{1}(\mathbb{R})$ with $f\ge 0$.
\end{proof}

 Observe, that the choice $(G,\beta)=(\T,\beta_{\T})$ and $\varphi=(\chi_{[0,2\pi]}-\chi_{[2\pi,4\pi]}) d\lambda$ leads to $\|\Phi(\varphi)\|=0<4\pi=\|\varphi\|$, so that in general $\Phi$ may fail to be isometric.
\medskip

As a first application of Lemma \ref{measuresembedd} we interpret the Poisson kernel on $\R$
 $$P_{u}\colon \R \to \R, ~ t\mapsto\frac{1}{\pi} \frac{u}{u^{2}+t^{2}}\,,~u>0,$$
as a measure on $G$. More precisely we define the 'Poisson measure' $p_{u}:=\Phi(P_{u}dt) \in M(G)$. Then $\|p_{u}\|=1$ and for $f\in L_{p}(G)$ we obtain $f*p_{u} \in L_{p}(G)$ with $\widehat{f*p_{u}}(h_{x})=\widehat{f}(h_{x})e^{-u|x|}$ for all $x \in \widehat{\beta}(\widehat{G})$. See Lemma~\ref{limituto0}, where this
convolution is needed.

\medskip

\subsection{Hardy spaces over  groups} The following definition is essential.

\begin{Defi} \label{defihardygroups}
Let $G$ be a compact abelian group $G$,  $E \subset \widehat{G}$, and $1\leq p \leq \infty$. Then we denote
    the subspace of all $f \in L_{p}(G)$ which have Fourier transforms $\widehat{f}:\widehat{G} \to \mathbb{C}$
    supported by $E$ (i.e. $\widehat{f}(\gamma)\ne0$ implies $\gamma \in E$) by $H_{p}^{E}(G)$.
 \end{Defi}

 Note that  $H_{p}^{E}(G)\subset L^{p}(G)$ is closed by continuity of the Fourier transform,
 and hence it is a Banach space. Additionally we define $M_{E}(G)$ as the subspace of all regular and bounded complex Borel measures $\mu \in M(G)$ such that $\widehat{\mu}(\gamma)\ne 0$ implies $\gamma \in E$. Further, we write $Pol_{E}(G):=\operatorname{span}_{\C}E$.

 \medskip

  We need the following two  tools (Proposition~\ref{densitytrigonometric} and Proposition~\ref{dualmapping1}) which follow from standard arguments from harmonic analysis. For the sake of completeness we give full proofs.

\begin{Prop} \label{densitytrigonometric}
Let $1\le p< \infty$, $G$ be a compact abelian group and $E\subset \widehat{G}$. Then  $Pol_{E}(G)$ is dense in $(H_{p}^{E}(G), \|\cdot\|_{p})$, and dense in $H_{\infty}^{E}(G)$ with respect to the weak star topology $w(L_{\infty}(G),L_{1}(G))$.
\end{Prop}
\begin{proof}
Let $f \in H_{p}^{E}(G)$. We claim, that there is a sequence $(u_{n}) \subset L_{\infty}(G)$ such that $(f*u_{n})$ converges to $f$ with respect to the norm of $H_{p}^{E}(G)$. Then, since all $f*u_{n}\in C(G)$,  using the fact that every continuous function in $H_{p}^{E}(G)$ can be approximated with respect to the sup-norm on $G$ by polynomials in $Pol_{E}(G)$  (see \cite[\S 8.7.3]{Rudin62}), the claim follows. Let $\varepsilon >0$. Since the translation map $T \colon G \to L_{p}(G), ~ y\mapsto f_{y}:=f(y-\cdot)$ is uniformly continuous, there is a zero-neighborhood $V$, such that $\|f-f_{y}\|_{p}< \varepsilon$ for all $y \in V$. Let $m$ be the Haar measure of $G$. Then $m(V)>0$ and $u_{V}:=\frac{1}{m(V)}\chi_{V} \in L_{\infty}(G)$ with $\|u_{V}\|_{1}=1$. Since
\begin{equation*}
(f*u_{V})(x)-f(x)=\int_{G} (f(x-y)-f(x))u_V(y) ~dy,
\end{equation*}
we obtain with Minkowski's inequality
\begin{align*}
\|f-f*u_{V}\|_{p} &= \left( \int_{G} \left| \int_{G} (f(x-y)-f(x)) u_{V}(y) dy \right|^{p} dx\right)^{\frac{1}{p}} \\ &\le \int_{G}\left( \int_{G} \left|(f(x-y)-f(x))u_{V}(y)\right|^{p} dx\right)^{\frac{1}{p}} dy\\&=\int_{G} \|f-f_{y}\|_{p} |u_{V}(y)| dy\le \varepsilon .
\end{align*}
 Further $\widehat{f*u_{V}}=\widehat{f} \,\,\widehat{u_{V}}$, which implies, that $f*u_{V} \in H_{p}^{E}(G)$.
  Finally, for the case $p = \infty$ note that by the bipolar theorem $$\overline{Pol_{E}(G)}^{w^{*}}=\left(Pol_{E}(G)^{\circ_{L^{1}}}\right)^{\circ_{L^{\infty}}}.$$  Then a direct calculation shows that $Pol_{E}(G)^{\circ_{L_{1}}}=H_{1}^{G\setminus E} (G)$ and $\left(H_{1}^{G\setminus E} (G)\right)^{\circ_{L_{\infty}}}=H_{\infty}^{E}(G)$ which finishes the proof.
 \end{proof}

For the proof of our second tool we need more abstract harmonic analysis.

\begin{Lemm} \label{kuhvomeis}

Let $1\le p \le \infty$, $G$ and  $H$  compact abelian groups, and $\phi \colon G \to H$  an injective continuous homomorphism with dense range. Then there are onto isometries
\begin{equation} \label{B1}
T_{p} \colon L_{p}(H) \to L_{p}(G) \,\,\, \text{ with $\widehat{f}=\widehat{T_{p}(f)} \circ \widehat{\phi}$}
\end{equation}
and
\begin{equation} \label{B2}
T \colon M(H)\to M(G) \,\,\, \text{ with $\widehat{\mu}=\widehat{T(\mu)} \circ \widehat{\phi}$\,.}
\end{equation}

\end{Lemm}
\begin{proof}
Since $G$ is compact, the dual group $\widehat{G}$ carries  the discrete topology, and so by Lemma \ref{dualmapping} the map $\widehat{\phi} \colon \widehat{H} \to \widehat{G}, ~ \gamma \mapsto \gamma \circ \phi$ is bijective. This implies that
$$\Phi_\infty \colon (C(H), \|\cdot\|_{\infty}) \to (C(G), \|\cdot\|_{\infty}), ~~ f \mapsto f\circ \phi$$ is an onto isometry
(to see this define  $\Phi_\infty$ first only for  polynomials and extend by density). Taking
the transposed of $\Phi_{\infty}^{-1}$, we by the Riesz representation theorem get the onto isometry
\begin{equation*} \label{OOO}
T:=(\Phi_{\infty}^{-1})':   M(H)=C(H)^{\prime} \to M(G)=C(G)^{\prime}\,,
\end{equation*}
which for all $\mu \in C(H)^{\prime}$ satisfies $\widehat{\mu}=\widehat{T(\mu)} \circ \widehat{\phi}$; indeed,
for $h \in \widehat{H}$
\[
\widehat{T(\mu)}(\widehat{\phi}(h)) = [(\Phi_{\infty}^{-1})'(\mu)](\overline{h \circ \phi})
= \mu(\Phi_{\infty}^{-1}(\overline{h} \circ \phi)) = \widehat{\mu}(h)\,.
\]
 This
proves \eqref{B2}. Next we check \eqref{B1} for $p=1$. Since  the two embeddings $i_G: L_{1}(G)\hookrightarrow M(G)$ and $i_H: L_{1}(H)\hookrightarrow M(H)$
 are isometric, we from another application of the fact that $\widehat{\phi}:\widehat{H} \to \widehat{G}$
is bijective and \eqref{B2} obtain that the map
\begin{equation*} \label{AAA} T_0: (Pol(H), \|\cdot\|_{1}) \to (Pol(G), \|\cdot\|_{1}), ~~ P \mapsto P\circ \phi
\end{equation*}
is an isometric bijection; for this check that
$$ (T \circ i_H)(h) =(i_G \circ T_0)(h)   \,\,\,\text{for each $h \in \widehat{H}$}\,.$$
 Additionally, we have
$\widehat{P \circ \phi}\circ \widehat{\phi}= \widehat{P}$ for all $P \in Pol(H)$,
and consequently by density of the polynomials we find as desired the isometric bijection
$T_1:L_{1}(H) \rightarrow L_{1}(G)$, which indeed satisfies
$\widehat{T_{1}(f)} \circ \widehat{\phi}= \widehat{f}$ for all $f\in L_{1}(H)$.
   In order to extend this result to all $L_p$'s note first that
\begin{equation} \label{zui}
\Phi_1: (C(H), \|\cdot\|_{1}) \mapsto (C(G), \|\cdot\|_{1}), ~~ f \mapsto f\circ \phi
\end{equation}
defines an onto isomorphism. Indeed, let us first check that $\Phi_1$ is an isometry. Take  $f \in C(H)$, and $P_{N} \in Pol(H)$ such that $\lim_{N} P_{N}=f$ in $C(H)$. Then $f\circ \phi=\Phi_1(f)=\lim_{N}\Phi_{\infty}(P_{N})=\lim_{N} P_{N}\circ \phi$ in $C(G)$ by continuity of
$\Phi_\infty$, and hence  in particular $\|f\circ \phi\|_{1}=\lim_{N} \|P_{N}\circ \phi\|_{1}=\lim_{N} \|P_{N}\|_{1}=\|f\|_{1}$. It remains to prove that $\Phi_1$ is onto. So let  $g \in C(G)$ and $Q_{N} \in Pol(G)$ such that $\lim_{N} Q_{N}=g$ in $C(G)$. Define $P_{N}:=\Phi_{\infty}^{-1}(Q_{N})$. Then $(P_{N})$ is Cauchy in $C(H)$ with limit, say, $f:=\lim_{N}P_{N}$. This implies that $\|f\|_{1}=\|g\|_{1}$, and we for all $x\in G$ have
$$f\circ \phi(x)=\lim_{N} P_{N}(\phi(x))=\lim_{N} Q_{N}(x)=g(x).$$
Now $(\ref{zui})$ in particular implies that
$$(Pol(H), \|\cdot\|_{p}) \mapsto (Pol(G), \|\cdot\|_{p}), ~~ P \mapsto P\circ \phi$$
is an onto isometry, and we  as desired  see that by density
of the polynomials
the onto isometry $T_p:L_{p}(H) \rightarrow L_{p}(G)$ for all $1\le p<\infty$ exists.
Finally, the remaining case $p =\infty$ follows by duality from the case $p=1$:
Define $T_\infty = [(T_1)']^{-1}: L_{\infty}(H) \rightarrow L_{\infty}(G) $.
\end{proof}

Moreover, we need the following fact from the theory of idempotent measures (see \cite[\S 3.1.2, p. 59 and \S 3.1.3, Theorem, p. 60]{Rudin62}).

\begin{Lemm} \label{idempotent} Let $G$ be a compact abelian group and $U \subset \widehat{G}$ a subgroup. Then there is some $\mu \in M(G)$ with  $\|\mu \|=1$ and $\widehat{\mu}=\chi_{U}$, where $\chi_{U}$ is the characteristic function on $U$.
\end{Lemm}

Finally, we are prepared to prove the announced  second tool; note that for injective maps $\phi$ and $E = \widehat{H}$ this is precisely Lemma \ref{kuhvomeis}.

 \begin{Prop} \label{dualmapping1}
Let $1\le p \le \infty$, $G$ and $H$ two compact abelian groups,  and  $\phi \colon G \to H$  a continuous  homomorphism with dense range. Then for all $E \subset \widehat{H}$ there are onto isometries
\begin{equation} \label{A1}
\Psi_p \colon  H_{p}^{E}(H)\to H_{p}^{\widehat{\phi}(E)}(G)
\,\,\, \text{ with $\widehat{f} =\widehat{\Psi_p(f)}\circ \widehat{\phi}$\,,}
\end{equation}
and
\begin{equation} \label{A2}
\Psi \colon M_{E}(H) \to M_{\widehat{\phi}(E)}(G)
\,\,\, \text{ with $\hat{\mu}=\widehat{\Psi(\mu)}\circ \widehat{\phi}$\,.}
\end{equation}
\end{Prop}
The construction of our proof shows that the mapping $\Psi$ from \eqref{A2} in fact is an extension 
of  the  mapping $\Psi_p$ from \eqref{A1}.

\begin{proof}
Since $\phi$ is continuous and has dense range, the map $$[\phi] \colon G/\ker \phi \hookrightarrow H, ~  [g] \mapsto \phi(g)$$
is injective, and it still has dense range.
Hence by Lemma \ref{kuhvomeis} we have that  $L_{p}(G/\ker \phi)=L_{p}(H)$ for all $1\le p\le \infty$ as well as $M( G/\ker \phi)=M(H)$, where the Fourier coefficients are preserved with respect to $\widehat{[\phi]}$,
and consequently
\[
H^{\widehat{[\phi]}(E)}_{p}(G/\ker \phi)=H^E_{p}(H)
\,\,\, \text{and}\,\,\, M_{\widehat{[\phi]}(E)}(G/\ker \phi)=M_E(H)\,.
\]
Moreover, we have the bijective  homomorphism
\[
\iota: \text{Im} \widehat{\phi} \hookrightarrow (G/\ker \phi)^{\widehat{}},\, \,\,
\widehat{\phi}(\gamma) \to \big[[g] \to \gamma(\phi( g))\big]
\]
(see also the  proof of Lemma \ref{dualmapping}, where we checked  that $(G/\ker \phi)^{\widehat{}} = (\ker \phi)^{\perp}=\overline{\text{Im} \widehat{\phi}} = \text{Im} \widehat{\phi}$, since $\widehat{H}$ carries the discrete
topology), and  hence it suffices  to show that
for every $K\subset \operatorname{Im} \widehat{\phi}$
\[
H_{p}^{K}(G)=H_{p}^{\iota (K)}(G/\ker \phi)
\,\,\, \text{and}\,\,\,
M_{K}(G)=M_{\iota (K)}(G/\ker \phi)\,.
\]
Let $\pi: G \to G/\ker \phi$ be the canonical surjection.
Since the Haar measure $m_{\pi}$ of $G_{1}/\ker \phi$ is the pushforward measure of $\pi$, the map
$$ L_{p}(G/\ker \phi) \to L_{p}(G), ~~ f \mapsto f\circ \pi$$
for all $1\le p \le \infty$ defines an into isometry. In particular, fixing $K\subset \operatorname{Im} \widehat{\phi}$, we obtain the isometric mapping
\begin{equation} \label{betty}
A_{p} \colon H_{p}^{\iota(K)}(G/\ker \phi) \to  H_{p}^{K}(G),~~ f \mapsto f\circ \pi.
 \end{equation}
 Since $Pol_{K}(G)$ is dense in $H_{p}^{K}(G)$ (Proposition \ref{densitytrigonometric}), we have that $A_{p}$ is onto for all  $1\le p<\infty$. For $p = \infty$ let $g \in H_{\infty}^{K}(G)$. Then, since $H_{\infty}^{K}(G) \subset H_{p}^{K}(G)$ for all $1 \leq p < \infty$, by what we have proved so far, there is $f\in \bigcap_{1 \leq p < \infty} H_{p}^{\iota (K)}(G/\ker \phi)$ such that $\|f\|_{p}= \|g\|_{p}\le \|g\|_{\infty}$ for all $1 \leq p < \infty$. Hence, $\lim_{p\to \infty} \|f\|_{p}\le \|g\|_{\infty}$, and so $f \in H_{\infty}^{\iota (K)}(G/\ker \phi)$, $A_{\infty}(f)=g$ and $\|f\|_{\infty}=\lim_{p\to \infty} \|f\|_{p}\le \|g\|_{\infty}$.  This completes the
proof of \eqref{A1}.

 For the proof of   \eqref{A2}  we use, given a compact abelian group $L$ and  $F \subset \widehat{L}$, the notation
 $$C(L)^{\prime}_{F}:=\{ \varphi \in C(L)^{\prime} \colon \varphi(\gamma)\ne 0 \Rightarrow \gamma \in F\}\,.$$
 Then by the Riesz representation theorem $M_{F}(L)=C(L)^{\prime}_{F}$.

 Fix again some $K\subset \operatorname{Im  \widehat{\phi}}$.
 Restricting the map $A_\infty$ in   \eqref{betty} to continuous functions, we conclude that the  map
 $$B_{\infty} \colon C_{\iota(K)}(G/\ker \phi) \to C_{K}(G), ~~ f \mapsto f\circ \pi$$
is also an onto isometry. Now  consider the mapping
$$\Phi \colon C(G)^{\prime}_{K}  \to C(G/\ker \phi)^{\prime}_{K},~~ \varphi \mapsto \varphi \circ B_{\infty}.$$
Then $\Phi$ is defined and isometric; indeed, since $\varphi$ is supported on $K$, we have
$$\|\Phi(\varphi)\|=\sup_{g \in C(G/\ker \phi), ~\|g\|_{\infty}=1} |\varphi(g\circ \pi)|=\sup_{f \in C_{K}(G), ~\|f\|_{\infty}=1} |\varphi(f)|=\|\varphi\|.$$
To prove surjectivity choose $\mu \in M(G)$ with $\|\mu\|=1$ and  $\widehat{\mu}=\chi_{\operatorname{Im \widehat{\phi}}}$, where $\chi_{\operatorname{Im} \widehat{\phi}}$ denotes the characteristic function on $\operatorname{Im} \widehat{\phi}$
(Lemma \ref{idempotent}). Then, given $\eta \in C(G/ \ker \phi)^{\prime}_{K}$, the functional
$\varphi \colon C(G)\to \C, \varphi(f):=\eta(f*\mu)$ belongs to  $C(G)^{\prime}_{K}$ and $\Phi(\varphi)=\eta$ (using again $C_{\iota (K)}(G/\ker \phi)=C_{\operatorname{Im} \widehat{\phi}}(G)$ via $B_{\infty}$). So $\Phi$ is an onto isometry, and now applying the Riesz representation theorem, the mapping $\Psi:=\Phi^{-1}$ fulfils the claim.
\end{proof}

\smallskip

\subsection{Hardy spaces over $\pmb{\lambda}$-Dirichlet groups}\label{Bohrcompact}
The following notion will turn out to be fundamental for our purposes.
\begin{Defi}
Let $\lambda$ be a frequency and $(G, \beta)$  a Dirichlet group. Then $G$ is called $\lambda$-Dirichlet group whenever $\lambda \subset \widehat{\beta}(\widehat{G})$.
\end{Defi}
So for instance for ordinary Dirichlet series, $\lambda_{n}=\log(n)$, the groups $\T^{\infty}$  and $\overline{\R}$ are $\log(n)$-Dirichlet groups. But more can be said.

\begin{Exam}
From Example \ref{examples1} we immediately deduce that the pair $(\overline{\R}, \beta_{\overline{\R}})$
forms a $\lambda$-Dirichlet group for any frequency $\lambda$.
\end{Exam}

We explain in Section \ref{severalvariables} that for every frequency $\lambda$ there is a suitable map $T_{\lambda}$ such that $(\widehat{\mathbb{Q}_d}^{\infty}, T_{\lambda})$ is  a $\lambda$-Dirichlet groups, which can be seen as an analog of $(\T^{\infty}, \beta_{\T^{\infty}})$ from the ordinary case
(see Example \ref{examples2}). Further in Section \ref{section4} we define classes of frequencies $\lambda$ (natural and integer type) for which $\T^{\infty}$ is a $\lambda$-Dirichlet group (including $\lambda=(\log n)$).

\medskip
Let $(G,\beta)$ be a $\lambda$-Dirichlet group and $1\le p \le \infty$. Then
$E =\{ h_{\lambda_n} \colon n \in \mathbb{N}\} \subset \widehat{G}$, and we define
\[
H_{p}^{\lambda}\left(G\right):= H_p^E(G).
\]
So $H_{p}^{\lambda}\left(G\right)$ is  the Banach space of all functions $f \in L_p(G)$ such that $\widehat{f}(\gamma)\neq 0$ only if $\gamma = h_{\lambda_n}$ for some $n$.  Note that, the definition of $H_{p}^{\lambda}(G)$ also depends on the choice of $\beta$, although our notation does not indicate this.

\begin{Theo}\label{ident2} Let $1\le p \le \infty$ and $(G,\beta)$  a $\lambda$-Dirichlet group. Then
\[
 H_{p}^{\lambda}(G) =  H_{p}^{\lambda}(\overline{\R})\,\,\, \text{ isometrically}.
\]
More precisely,
there is an onto isometry $ \Psi \colon H_{p}^{\lambda}(G) \to H_{p}^{\lambda}(\overline{\R}),  f \mapsto g$ such that $\widehat{f}(h^{(G,\beta)}_{\lambda_{n}})=\widehat{g}(h^{(\overline{\R}, \beta_{\overline{\R}})}_{\lambda_{n}})$ for all $n \in \N$.
\end{Theo}

\begin{proof} The result is immediate from Proposition~\ref{biggestDirichletgroup} and Proposition~\ref{dualmapping1}.
\end{proof}

\begin{Coro}  \label{ident3} Let $\lambda$ be a  frequency  and $(G,\beta_{G})$ and $(H,\beta_{H})$  two $\lambda$-Dirichlet groups.
Then   $$H_{p}^{\lambda}(G)=H_{p}^{\lambda}(H)\,\,\, \text{ isometrically},$$ where, if $f \in H_{p}^{\lambda}(G)$ and $g \in H_{p}^{\lambda}(H)$
are associated to each other, we have that  $\widehat{f}(h_{\lambda_{n}}^{(G,\beta_{G})})=\widehat{g}(h_{\lambda_{n}}^{(H,\beta_{H})})$ for all $n \in \N$.
\end{Coro}

In the next result we show that every $f \in H_{\infty}^{\lambda}(G)$  in a natural way defines  holomorphic functions on the open right half plane; in Section \ref{fourier} we take advantages of this important feature.
\begin{Prop} \label{holomorphy} Let $\lambda$ be an arbitrary frequency and $f \in H_{\infty}^{\lambda}(G)$. Then for almost all $\omega \in G$ the function
$$
F_\omega\colon [Re>0] \to \C\,,\,\,\,F_\omega(u+it):=(f_{\omega}*P_{u})(t)
$$
defines a holomorphic function on $[Re>0]$ which is bounded by $\|f\|_{\infty}$.
\end{Prop}

\begin{proof}
By Lemma \ref{restrictiontoreals} we know that $f_{\omega} \in L_{\infty}(\R)$ for almost all $\omega \in G$ (say, for all $\omega$ not contained in a null set $M\subset G$). So $F_{\omega}$ defines a continuous function on $[Re>0]$ (actually a harmonic function on $[Re>0]$, see \cite{helsonharmonic}). For polynomials $g= \sum_{n=1}^{N} b_{n} e^{-i\lambda_{n}\cdot}$ we have that
for all $x \in \R$ and $u>0$ by Abel summation
\begin{equation*}
\mathcal{F}_{L_{1}(\R)}\left(e^{-u|\cdot|}\sum_{\lambda_{n}<\cdot} b_{n}\right)(x)=\frac{1}{u+ix} \sum_{n=1}^{N} b_{n}e^{-(u+ix)\lambda_{n}}=\frac{(g*P_{u})(x)}{u+ix}  \,.
\end{equation*}
Since the right hand side (as a function of $x$) is in $L_{2}(\R)$, we apply the Fourier transform on $L_{2}(\R)$ and obtain
\begin{equation}\label{fourier1}
\mathcal{F}_{L_{2}(\R)} \left(\frac{g*P_{u}}{u+i\cdot}\right)(-\cdot)=e^{-u|\cdot|} \sum_{\lambda_{n}<\cdot} b_{n}\,\,\, \text{in $L_{2}(\R)$}.
\end{equation}
 Now we want to apply this formula for $f_{\omega}$. Therefore fix $u>0$ and consider the mapping
\[
\Psi_u: L_2(G)  \to L_2(G, L_2(\mathbb{R})), \,\, g \mapsto \Big[\omega \mapsto \frac{g_\omega*P_{u}}{u+i \cdot}\Big]\,.
\]
This operator is well-defined and bounded. Since the sequence  $f^{N}:=\sum_{n=1}^{N} a_{n} h_{\lambda_{n}}$ converges to $f$ in $L_{2}(G)$, we obtain some  subsequence $(N_{k})_{k}$  such that for almost all $\omega \in G$, say $\omega \notin N_{u} \cup M$, where $N_{u}\subset G$ is a null set,
\begin{equation*} \label{fourier4}
\lim_{k\to \infty}\frac{f^{N_{k}}_{\omega}*P_{u}}{u+i\cdot}=\frac{f_{\omega}*P_{u}}{u+i\cdot}
\,\,\, \text{in $L_{2}(\R)$}\,.
\end{equation*}
Then by the continuity of the Fourier transform on $L_{2}(\R)$ and \eqref{fourier1} we obtain
for all $\omega \notin N_{u} \cup M$
\begin{equation} \label{fourier5}
\mathcal{F}_{L_{2}(\R)}\left(\frac{f_{\omega}*P_{u}}{u+i\cdot}\right)(-\cdot)=e^{-u|\cdot|}\sum_{\lambda_{n}<\cdot} a_{n} \omega(\lambda_{n}) \in L_{2}(\R)\,.
\end{equation}
In particular for all $u>0$ and $\omega \notin N_{u} \cup M$
\begin{align*}
&\left\|e^{-2u|\cdot|}\sum_{\lambda_{n}<\cdot} a_{n} \omega(\lambda_{n})\right\|_{L_{1}(\R)}=\int_{0}^{\infty} e^{-2ux} \left|\sum_{\lambda_{n}<x} a_{n} \omega(\lambda_{n})\right| dx \\ &\le \left(\int_{0}^{\infty} \left|\left(\sum_{\lambda_{n}<x} a_{n} \omega(\lambda_{n}) \right) e^{-ux}\right|^{2} dx\right)^{\frac{1}{2}} \left(\int_{-\infty}^{\infty} e^{-2u|x|} dx \right)^{\frac{1}{2}}< \infty,
\end{align*}
hence $e^{-u|\cdot|}\sum_{\lambda_{n}<\cdot} a_{n} \omega(\lambda_{n})\in L_{1}(\R)$ for all $\omega \notin N_{u} \cup M$. Now we write $\mathbb{Q}_{+}:=\mathbb{Q}\cap ]0,\infty[=\{q_{1}, q_{2}, \ldots\}=(q_{n})_{n}$ and choosing $u_{n}=q_{n}$ we obtain a null set $N:=\bigcup_{n\in \N} N_{q_{n}}\subset G$ such that $e^{-u|\cdot|}\sum_{\lambda_{n}<\cdot} a_{n} \omega(\lambda_{n})\in L_{1}(\R)$ for all $u\in \mathbb{Q}_{+}$ and  $\omega \notin N\cup M$. So for all $\omega \notin N\cup M$ the function
\[
g_\omega: [\text{Re} >0] \to \mathbb{C}\,, \,\, g_\omega(z) =
\int_{0}^{\infty} \left(\sum_{\lambda_{n}<x} a_{n} \omega(\lambda_{n}) \right)e^{-zx} dx
\]
is well-defined and holomorphic. Now together with (\ref{fourier5}) we obtain for all $u\in \mathbb{Q}_{+}$ and  $\omega \notin N\cup M$
\begin{align*}
&\int_{0}^{\infty} \Big(\sum_{\lambda_{n}<x} a_{n} \omega(\lambda_{n}) \Big)e^{-(u+i\cdot)x} dx=\mathcal{F}_{L_{1}(\R)}\left(e^{-u|\cdot|}\sum_{\lambda_{n}<\cdot} a_{n} \omega(\lambda_{n})\right)\\ &=\mathcal{F}_{L_{2}(\R)}\left(e^{-u|\cdot|}\sum_{\lambda_{n}<\cdot} a_{n} \omega(\lambda_{n})\right)=\mathcal{F}_{L_{2}(\R)}\left(\mathcal{F}_{L_{2}(\R)} \left( \frac{f_{\omega}*P_{u}}{u+i\cdot} \right)(-\cdot)\right)=\frac{f_{\omega}*P_{u}}{u+i\cdot}.
\end{align*}
 So finally for all $u\in \mathbb{Q}_{+}$  and for all $\omega \notin N\cup M$ and almost all $t$
\begin{equation} \label{fourier3}
\frac{F_{\omega}(u+it)}{u+it}=\int_{0}^{\infty} \Big(\sum_{\lambda_{n}<x} a_{n} \omega(\lambda_{n}) \Big)e^{-(u+it)x} dx=g_\omega(u+it)\,.
\end{equation}
Since both sides are continuous functions, equation (\ref{fourier3}) holds on $[Re>0]$. So since \[
F_\omega(z) = z g_\omega(z)\,\, \text{on $[\text{Re}>0]$}\,,
\]
 $F_\omega$ is indeed holomorphic on the open right half plane for all $\omega \notin N\cup M$.
\end{proof}

\subsection{Hardy spaces  of general Dirichlet series} \label{hardyspace}
We  now come to the actual goal of this work.
Let $1\le p \le \infty$ and $(G, \beta)$ be a $\lambda$-Dirichlet group, and consider the following map which we call Bohr map:
\begin{equation*}
\Bcal \colon H^{\lambda}_{p}(G) \hookrightarrow \mathcal{D}(\lambda),
\,\,\, f \sim \sum_{\gamma \in \widehat{G}} \widehat{f}(\gamma) \gamma\,\,\, \mapsto\,\,\, \sum_{n \in \N} \widehat{f}(h_{\lambda_{n}}) e^{-\lambda_{n}s}.
\end{equation*}
The following definition is at the very heart of this work.
\begin{Defi} \label{zugcrash}
The Hardy space $\Hcal_{p}(\lambda)$ of $\lambda$-Dirichlet series
is the  space
of all  $\sum a_n e^{-\lambda_n s}$ for which there is some
$f \in H_p^\lambda(G)$ such that   $a_n = \widehat{f}(h_{\lambda_{n}})$ for all $n$.
\end{Defi}

Together with
the norm $\|D\|_{p}:=\|f\|_{p}$ the space $\Hcal_{p}(\lambda)$ clearly forms a Banach space, and then by definition the Bohr map $\Bcal$ gives an isometric onto isomorphism
from $\Hcal_{p}(\lambda)$ onto $H_{p}^{\lambda}(G)$.
The following immediate consequence of Corollary \ref{ident3}  is crucial.

\begin{Theo}
$\Hcal_{p}(\lambda)$  is independent of the chosen $\lambda$-Dirichlet group $(G, \beta)$.
\end{Theo}

Recall that in the ordinary case $\lambda:=(\log n)$, the groups $\T^{\infty}$ and $\overline{\R}$ are suitable $(\log n)$-Dirichlet groups, which immediately leads to the following consequence.
\begin{Coro} \label{oorrddii}
$$H_{p}^{\log(n)}(\overline{\R})=\mathcal{H}_{p}(\log(n))=H_{p}(\T^{\infty}).$$
\end{Coro}
So the above definition of $\Hcal_{p}(\lambda)$  actually coincides with Bayart's definition of $\Hcal_{p}$ in the ordinary case (see \cite{Bayart}).

\medskip

 Dealing with concrete problems in $\mathcal{H}_{p}(\lambda)$ it suffices to restrict ourselves thinking on $\overline{\R}$ (or $\widehat{\mathbb{Q}_d}^{\infty}$, see Section \ref{severalvariables}) instead of considering an 'abstract' Dirichlet group $G$. The choice of the group may depend then on the problem we are interested in. Of course considering particular frequencies like $\lambda=(n)$ or $\lambda=(\log n)$ it has advantages to work with the torus instead of $\overline{\R}$ due to its connection to complex analysis (on products of the unit disk $\D$).

\medskip

Note that now there are two  `$H_\infty$-spaces of $\lambda$-Dirichlet series' around, namely $\mathcal{D}_{\infty}(\lambda)$ and $\Hcal_{\infty}(\lambda)$. In Section \ref{fourier} we show that they coincide if $\mathcal{D}^{ext}_{\infty}(\lambda)=\mathcal{D}_{\infty}(\lambda)$ and $L(\lambda)<\infty$ (Theorem~\ref{mainresultinfty}),  which answers the question posed in the introduction.
\medskip

We close this section by giving an internal description of $\Hcal_{p}(\lambda)$ through $\lambda$-Dirichlet polynomials without considering Dirichlet groups. We denote by $Pol(\lambda)$ the space of all $\lambda$-Dirichlet polynomials $D=\sum_{n=1}^{N} a_{n} e^{-\lambda_{n}s}$. For such polynomials we define
\begin{equation*} \label{limitHp}
\|D\|_{p}:=\left(\lim_{T \to \infty} \frac{1}{2T} \int_{-T}^{T}\left|\sum^{N}_{n=1} a_{n}e^{-\lambda_{n}it} \right|^{p} dt\right)^{\frac{1}{p}}\,,
\end{equation*}
where due to  Proposition \ref{basic}  this limit  exists and gives a norm on  $Pol(\lambda)$. Then the following result is  an immediate consequence of
Proposition~\ref{densitytrigonometric}.
\begin{Theo} \label{comp_pol}Let $1\le p < \infty$ and $\lambda$ be a frequency. Then the space $\Hcal_{p}(\lambda)$ is the completion of $(Pol(\lambda), \|\cdot\|_{p})$.
\end{Theo}

\subsection{Several variables} \label{severalvariables}

As already mentioned in Corollary~\ref{oorrddii}, in the ordinary case $\lambda=(\log n)$ the $\Hcal_{p}$-spaces coincide with the Hardy space $H_{p}(\T^{\infty})$ which consists of functions in infinitely many variables. On the other hand by Theorem \ref{mainresultinfty} ordinary Dirichlet series can be understood in terms of the 'one dimensional' object  $\overline{\R}$. So the (vague) question might arise where are all the variables gone in the general case?

\medskip

Recall the following notions going back to  Bohr (see \cite[\S 5]{Bohr3}). An infinite matrix $R:=(r^{n}_{k})_{n,k \in \mathbb{N}}$ of rational numbers is called Bohr matrix
whenever each row  $R_n = (r_k^n)_k$ is finite, i.e. $r^{n}_{k} \neq 0$ for only finitely many $k$'s.
Given  a sequence $\lambda:=(\lambda_{n})$ of real numbers,  a sequence $B:=(b_{n})$ in $\R$ is said to be a basis for $\lambda$ if it is  $\mathbb{Q}$-linearly independent and for  each $n$ there is a finite sequence  $(r^{n}_{k})_{k}$ of  rational coefficients
 such that $\lambda_{n}=\sum_k r^{n}_{k} b_{k}$. In this case, $R:=(r^{n}_{k})_{n,k}$ is said to be a  Bohr matrix of $\lambda$ with respect to the basis $B$. If $\lambda$ is a frequency, such a basis always exists
 (in fact it  can be chosen as a subsequence of $\lambda$), and if $R$ is the associated  Bohr matrix $R$, we write  $\lambda=(R,B)$.

 \medskip
 Note that $(\log p_{j})$ is a basis of $\lambda=(\log n)$, where $\mathfrak{p}=(p_{j})$ is the sequence of prime numbers. In this case all finite sequences of natural numbers form the rows of
 $R$.

 \medskip

We come back to Example \ref{examples3} within the setting of $\lambda$-Dirichlet groups.

\begin{Exam}\label{examples8} \text{}
Let $B:=(b_{n})$  be a $\mathbb{Q}$-linearly independent  sequence of length
 $N \in \N \cup \{\infty\}$.
Then for every frequency $\lambda = (R, B)$ we have that $\big(\widehat{\mathbb{Q}_d}^{N}, \widehat{T_{B}}\big)$ is a $\lambda$-Dirichlet group, where $T_B$ is the mapping defined in \eqref{themap}.
\end{Exam}

On the other hand, if $R$ is a Bohr matrix and all rows $R_n$ have length $N \in \N \cup \{\infty\}$, then obviously each of these rows belongs to  the dual of $\widehat{\mathbb{Q}_d}^{N}$\,,
 and we denote for  $1\le p \le \infty$ the corresponding Hardy space (see Definition~\ref{defihardygroups}) by
 \[
 H^{R}_{p}\big(\widehat{\mathbb{Q}_d}^{N}\big)\,.
 \]
 Then Example~\ref{examples8} and Theorem \ref{ident2} show that in this case the three objects $H_{p}^{R}\big(\widehat{\mathbb{Q}_d}^{N}\big)$, $\Hcal_{p}(\lambda)$, and $H_{p}^{\lambda}(\overline{\R})$
 can not be distinguished in terms of Fourier- and Dirichlet coefficients.

  \begin{Theo} \label{summary} Let $1\le p\le \infty$ and $R$ be a Bohr matrix. Then for for every frequency  $\lambda=(R,B)$ with a  basis of length $N\in \N \cup \{\infty\}$
$$ H_{p}^{R}\big(\widehat{\mathbb{Q}_d}^{N}\big)=\Hcal_{p}(\lambda)=H_{p}^{\lambda}(\overline{\R})\,\,\, \text{isometrically}\,.$$
More precisely, there are unique onto isometries between these spaces such that if $F$, $D$, and $f$ are associated to each other, then
for each $n$
$$\widehat{F}(R_{n})=a_{n}(D)=\widehat{f}(\lambda_{n})\,.$$
\end{Theo}

Thinking of multi indices as rows in the Bohr matrix of $(\log n)$ the Hardy type space $H_{p}^{R}(\widehat{\mathbb{Q}_d}^{N})$ could be considered as the 'multi variable' substitute of $H_{p}(\T^{\infty})$ in the general case. The difference of $H_{p}^{R}(\widehat{\mathbb{Q}_d}^{N})$ and $H_{p}^{\lambda}(\overline{\R})$ lies in the kind of perspective: Either the focus lies on the frequency $\lambda$ itself or on its coefficients with respect to some decomposition $(R,B)$. Calculations are sometimes easier in $H_{p}^{\lambda}( \overline{\R})$ than in $H_{p}^{R}(\widehat{\mathbb{Q}_d}^{N})$, since we do not have to deal with 'infinitely many variables'. On the other hand the spaces $H_{p}^{R}(\widehat{\mathbb{Q}_d}^{N})$ decompose the frequency with respect to the basis ('separates the variables'), and they are independent of a chosen basis for $\lambda$. Moreover $\widehat{\mathbb{Q}_d}^{N}$ is metrizable (since its dual group is countable).

\medskip

An immediate consequence of Theorem \ref{summary} is the following (which may be is not immediate considering $H_{p}^{\lambda}(\overline{\R})$).
\begin{Coro} \label{samecoeff} Assume that  $1\le p \le \infty$ and that  $\lambda=(R,B)$ and $\widetilde{\lambda}=(\widetilde{R}, \widetilde{B})$ are two frequencies. If there is a permutation  $\phi \colon \N \to \N$ such that $\widetilde{R_{n}}=R_{\phi(n)}$ for all $n \in \N$, that is $\lambda$ and $\widetilde{\lambda}$ have the same Bohr matrix without regard to the order of the rows, then $\Hcal_{p}(\lambda)=\Hcal_{p}(\widetilde{\lambda})$ (as Banach spaces) and the coefficient are preserved with respect to $\phi$.
\end{Coro}
For instance choosing $\lambda:=(\log 2^{n})$ and $\widetilde{\lambda}=(\log 3^{n})$, then  Corollary \ref{samecoeff} implies $\mathcal{D}_{\infty}(\lambda)=H_{\infty}(\T)=\mathcal{D}_{\infty}(\widetilde{\lambda})$ (which of course can be deduced directly).

\smallskip

\subsection{Frequencies of  integer type} \label{section4}
In this section we consider frequencies which allow a decomposition $(R,B)$ such that the Bohr matrix $R$ consists exclusively of integers or natural numbers. We call a Bohr matrix of integers or natural numbers  full, if
any possible finite sequence of integers or  numbers in $\mathbb{N}_0$ appears as a row. This includes the Bohr matrix $R$ generated by  the ordinary case $\lambda=(\log n)$ and its basis $B=(\log p_n)$.
 If $\alpha$ is a row of $R$, then we shortly write $\alpha \in R$. For $1 \leq p \leq \infty$ we keep to the standard denoting by $\Hcal_{p}$ the Hardy space of ordinary Dirichlet series.
\begin{Defi} We call a frequency $\lambda$ of integer (resp. natural) type if there is a basis $B$ for $\lambda$ such that the Bohr matrix $R$ only consists of integer (resp. natural) numbers.
\end{Defi}
\noindent
So for these frequencies $\lambda = (R,B)$ of integer type the group $\T^{N}$  is a $\lambda$-Dirichlet group
(use Example~\ref{examples2}), where $N \in \N \cup \{\infty\}$ is the length of the basis $B$, and we obtain the following
\begin{Theo} \label{integertype} Let $1\le p \le \infty$ and  $\lambda=(R,B)$  a frequency of integer type with basis of length $N \in \N\cup\{\infty\}$. Then there is a unique onto isometry $$\psi \colon \Hcal_{p}(\lambda) \to H_{p}^{R}(\T^N),  ~~ D \mapsto F$$ such that $\hat{F}( \alpha)=a_{n}(D)$ for all multi indices $\alpha \in  R$ and $\lambda_{n}=\sum \alpha_{j} b_{j}$.
\end{Theo}
Theorem \ref{integertype} also implies that the ordinary $\Hcal_{p}$ is in the following sense the 'biggest' Hardy type space of natural type.
\begin{Coro} \label{ordinary} Let $1\le p \le \infty$ and  $\lambda=(R,B)$  a frequency of natural type. Then there is a unique into isometry
$$ \psi \colon \Hcal_{p}(\lambda) \hookrightarrow \Hcal_{p}, ~~ \sum a_{n} e^{-\lambda_{n}s} \to \sum b_{n} n^{-s}$$
such that $a_{n}=b_{\mathfrak{p}^{\alpha}}$ for all $\alpha \in R$, where $\lambda_{n}=\sum \alpha_{j} b_{j}$.
\end{Coro}

So, if $\lambda$ is of integer or natural type, then the group $\mathbb{T}^\infty$ is appropriate. The following remark can be seen as a sort of converse.

\begin{Rema} If $(\T^{\infty}, \beta)$ is a $\lambda$-Dirichlet group, then $\lambda$ is of integer type.
\end{Rema}
\begin{proof}
By assumption and Lemma~\ref{dualmapping} there is a monomorphism $T\colon \oplus_{n=1}^{N} \mathbb{Z} \hookrightarrow \R$ such that $\lambda \subset \text{Im} ~T$. Denoting by $e_{k}$ the $k$th unit vectors in $\oplus_{n=1}^{N} \mathbb{Z}$ the sequence $(T(e_{k}))_{k}$ is $\mathbb{Z}$-linearly independent and $\lambda \subset \operatorname{span}_{\mathbb{Z}} (T(e_{k}))_{k}$.
\end{proof}

\subsection{On Bohr's problem} \label{B-prob}

A prominent problem in Bohr's time was to determine the maximal width of the strip of uniform but non absolutely convergence
$$S(\lambda):=\sup_{D \in \mathcal{D}(\lambda)} \sigma_{a}(D)-\sigma_{u}(D).$$
Note that we always have $S(\lambda)\le \frac{L(\lambda)}{2}$ by the Cauchy-Schwarz inequality. In the ordinary case, where $L(\log(n))=1$, a celebrated result is that $S(\log(n))=\frac{1}{2}$ (see \cite[\S 5, Theorem V]{BohnenblustHille}, \cite{Defant}, \cite{Helson}, or \cite{QQ}).
So
the question may arise whether   $S(\lambda)=\frac{L(\lambda)}{2}$ always holds. But this is in general false due to a result of Neder
(see \cite[\S4]{Neder}, actually $\sigma_{b}^{ext}$, see \cite{DefantSchoolmann2} for definition, is considered instead of $\sigma_{u}$, but Neder's construction can easily be modified  for $\sigma_{u}$).
\begin{Prop} To every $x>0$ and $0\le y\le\frac{x}{2}$ there is a frequency $\lambda$ such that $L(\lambda)=x$ and $S(\lambda)=y$.
\end{Prop}
If $\lambda$ satisfies Bohr's theorem, then we may reformulate $S(\lambda)$  as follows
$$S(\lambda)=\sup_{ D\in \mathcal{D}_{\infty}(\lambda)} \sigma_{a}(D),$$
which actually in a sense is  one of the key steps for proving $S((\log n))=\frac{1}{2}$.
This motivates  to define (like already done in the ordinary case, see \cite[\S 12.2]{Defant}) the following '$p$-version' of $S(\lambda)$:
$$S_{p}(\lambda):=\sup_{D \in \mathcal{H}_{p}(\lambda)} \sigma_{a}(D).$$
Note that $S_{p}(\lambda)$ is decreasing in $p$ and that $S_{1}(\lambda)\le L(\lambda)$. In the particular case $p=2$ we have $S_{2}(\lambda)=\frac{L(\lambda)}{2}$ for all $\lambda$'s. Indeed, the Cauchy-Schwarz inequality gives $S_{2}(\lambda)\le\frac{L(\lambda)}{2}$, and conversely $D_{\varepsilon}(s)=\sum e^{-\left(\frac{L(\lambda)}{2}+\varepsilon\right)\lambda_{n}}e^{-\lambda_{n}s} \in \mathcal{H}_{2}(\lambda)$ satisfies $\sigma_{a}(D_{\varepsilon})=\frac{L(\lambda)}{2}-\varepsilon$ for all $\varepsilon>0$.

\medskip

Given $\lambda$, the exact value of $S_{p}(\lambda)$ seems hard to determine. But at least if $\lambda$ is of full natural type, then we will show that
$S(\lambda)=\frac{L(\lambda)}{2}$ as a consequence of  Theorem \ref{ordinary}. For $\mathbb{Q}$-linearly independent frequencies see Corollary \ref{Qlin}.

\begin{Coro}\label{striptease}
 Let  $\lambda=(R,B)$ be a frequency of natural type, where $R$ is full. If  $\lim_{n\to \infty} \frac{\log(n)}{\lambda_{n}}$ exists, then
$S_{p}(\lambda)=\frac{L(\lambda)}{2}$ for all $1\le p\le \infty$. Further, if $\lambda$ satisfies $(LC)$ and $L(\lambda)<\infty$, then there is a Dirichlet series $D \in \mathcal{D}_{\infty}(\lambda)$ with $\sigma_{a}(D)=\frac{L(\lambda)}{2}$ and  $S(\lambda)=\frac{L(\lambda)}{2}$.
\end{Coro}
Maximal width of uniform and but non absolute convergence of Dirichlet series for frequencies $\lambda$ as in the previous corollary were also studied in the recent article \cite{Maestre}. There the authors
look at the special case $S(\lambda)= S_\infty(\lambda)$ assuming $(BC)$, but  without assuming that  $\lim_{n\to \infty} \frac{\log(n)}{\lambda_{n}}$ exists. Their reasoning  is based on the fact the set $\operatorname{mon} H_{\infty}(B_{c_{0}})$ of monomial convergence is included in $\ell_{2+\varepsilon}\cap \D^{\N}$ for all $\varepsilon>0$, which is basically a consequence of a probabilistic argument (Kahane-Zygmund inequality). Their proof extends word by word for $\lambda$'s satisfying $(LC)$. Using the fact that $\operatorname{mon} H_{p}(\T^{\infty})=\ell_{2}\cap \D^{\N}$, $1\le p <\infty$, and $\mathcal{H}_{p}(\lambda)=H_{p}(\T^{\infty})$ (see \cite[\S 12.4.2]{Defant}) the same argument gives $S_{p}(\lambda)=\frac{L(\lambda)}{2}$  without any assumption on the  $\lambda$'s. We feel that our proof  of Corollary~\ref{striptease} (under the stronger assumption that  $\lim_{n\to \infty} \frac{\log(n)}{\lambda_{n}}$ exists) is  more elementary.
\begin{proof}[Proof of Corollary \ref{striptease}] By Corollary \ref{ordinary}   we have that $\mathcal{H}_{p}(\lambda)=\mathcal{H}_{p}((\log n))$, where the isometry preserves the Dirichlet coefficients. Then by assumption for all $\varepsilon>0$ there is $n_{0}$ such that $\lambda_{n}(L(\lambda)-\varepsilon)\le \log(n) \le \lambda_{n}(L(\lambda)+\varepsilon)$ for all $n\ge n_{0}$. From this by an easy calculation we see that $\sigma_{a}(D)=\sigma_{a}(E)L(\lambda)$ for all $D \in \mathcal{H}_{p}(\lambda)$ and $E \in \mathcal{H}_{p}((\log n))$ associated to each other. This gives $S_{p}(\lambda)=L(\lambda) S_{p}((\log n))=\frac{L(\lambda)}{2}$. By Theorem \ref{Dstar} and Theorem \ref{mainresultinfty} (applied to $(\log n)$)  we have $\mathcal{D}_{\infty}(\lambda)\subset \mathcal{D}_{\infty}((\log n))$
since $\lambda$ is of natural type. Assuming $(LC)$ and $L(\lambda)<\infty$ we actually have equality (again Theorem \ref{mainresultinfty}), and we obtain $S(\lambda)=S((\log n))L(\lambda)=\frac{L(\lambda)}{2}$. Since there is $E \in \mathcal{D}_{\infty}((\log n))$ with $\sigma_{a}(E)=\frac{1}{2}$ (see \cite[\S 4]{Defant}), the claim is proven.

\end{proof}
Now we consider the case of $\mathbb{Q}$-linearly independent frequencies $\lambda$. Note that in this case we already know from \cite{DefantSchoolmann2} that $\mathcal{H}_{\infty}(\lambda)=\mathcal{D}_{\infty}(\lambda)=\ell_{1}(\N)$ and therefore that $S_{\infty}(\lambda)=S(\lambda)=0$.
\begin{Coro} \label{Qlin}
Let $1\le p  < \infty$ and $\lambda$ be $\mathbb{Q}$-linearly independent.Then the following isometric equalities hold:
\begin{equation*}
\mathcal{H}_{p}(\lambda)=H^{(1)}_{p}(\T^{\infty}):=\{ f \in H_{p}(\T^{\infty}) \mid \widehat{f}(\alpha)\ne 0 \Rightarrow \sum \alpha_{j}=1 \}
\end{equation*}
and
\begin{equation*}
 \mathcal{D}_{\infty}(\lambda)=H^{(1)}_{\infty}(\T^{\infty})\,.
 \end{equation*}
In particular, $S_{p}(\lambda)=\frac{L(\lambda)}{2}.$
\end{Coro}
\begin{proof}By Theorem \ref{ordinary} we know that  $\mathcal{H}_{p}(\lambda)=\Bcal(H^{(1)}_{p}(\T^{\infty}))=:\mathcal{H}^{(1)}_{p}((\log n))$ is the space of $1$-homogeneous ordinary Dirichlet series. Since $\mathcal{H}^{(1)}_{p}((\log n))=\mathcal{H}_{2}((\log n))$ isomorphically (by Khinchine's inequality, see \cite[\S 11.2.2]{Defant}) we have $\mathcal{H}_{p}(\lambda)=\mathcal{H}_{2}(\lambda)$ and so $S_{p}(\lambda)=S_{2}(\lambda)=\frac{L(\lambda)}{2}$.
\end{proof}

We close this section by giving an example of $\lambda$ which is not of integer type.
\begin{Exam}
Let $(p_{n})$ be the sequence of prime numbers and define $\lambda_{1}:=1$, $\lambda_{n}:=\frac{p_{n}^{2}+1}{p_{n}}$, $n \ge 2$. Then clearly $B:=(1, 0, \ldots)$ is a basis with Bohr matrix
$$R:=\begin{pmatrix}
 1 &0&\ldots \\
\frac{2^2+1}{2}&0& \ldots\\
\frac{3^2+1}{3}&0& \ldots\\
\frac{5^2+1}{5}&0& \ldots\\
\vdots &    \vdots & & \\
\end{pmatrix}.
$$
Assume that $\lambda$ is of integer type. Then there is a basis $X=(x, 0,\ldots)$ of length $1$ with an integer Bohr matrix (see \cite{Bohr3}, \S 5, pp. 121-122). Choose $(k_{n})\subset \Z \setminus \{0\}$ such that
$$\frac{p_{n}^{2}+1}{p_{n}}=k_{n}\,x$$
for all $n \in \N$. Hence $x \in \mathbb{Q}$. Now write $x=\frac{a}{b}$, where $a,b\ne 0$. Then
$$ b (p_{n}^{2}+1)=p_{n} k_{n} a.$$
Since $p_{n}^2+1$ is not divisible by $p_{n}$, we conclude  that $b$ is divisible by $p_{n}$ for all $n \in \N$.  Hence $b=0$, a contradiction. Moreover,
$$\lambda_{n+1}-\lambda_{n}=p_{n+1}-p_{n}-\left(\frac{1}{p_{n}}-\frac{1}{p_{n+1}}\right)\ge 1-\frac{1}{2}=\frac{1}{2}\,,$$
implying that $\lambda$ is strictly increasing and satisfies $(BC)$ with $l=0$.

\end{Exam}

\smallskip

\section{Structure theory} \label{structuretheory}
Based on the Fourier analysis setting from the preceding section we extend some important cornerstones of the  $\mathcal{H}_p$-theory of ordinary Dirichlet series to the range of general Dirichlet series.

\subsection{Translations} \label{hard}
In this section we fix a frequency $\lambda$, and collect a few independently interesting tools on translations of $\lambda$-Dirichlet series which  also later will be important.
\begin{Defi} Let $D=\sum a_{n}e^{-\lambda_{n}s}$ be a  $\lambda$-Dirichlet series and $z \in \C$. Then we call
$$D_{z}(s):=\sum a_{n} e^{-\lambda_{n}z} e^{-\lambda_{n}s}$$ the   translation of $D$ about $z$.
\end{Defi}
We call the translation $D_{z}$ horizontal if $z \in \R$, and vertical if $z=i\tau$, $\tau \in \R$.
 Given a $\lambda$-Dirichlet group $(G, \beta)$ and a  Dirichlet series $D=\sum a_{n}e^{-\lambda_{n}s}$, we for every  $\tau \in \R$  have
\[
D_{i \tau} =  \sum a_n e^{-i \lambda_n \tau}e^{-\lambda_n s} = \sum a_n h_{\lambda_n}(\beta(\tau))e^{-\lambda_n s}\,.
\]
More generally, we consider so-called vertical limits.
\begin{Defi} Given a $\lambda$-Dirichlet group $(G, \beta)$ and a Dirichlet series $D = \sum a_n e^{-\lambda_n s}$, we call  Dirichlet series of the form $$D^{\omega}(s)=\sum a_{n}\omega(\lambda_{n}) e^{-\lambda_{n}s},~ \omega \in G$$  vertical limits of $D$ associated to $G$.
\end{Defi}

In the following we study in which sense $\lambda$-Dirichlet series belonging to $\mathcal{H}_p(\lambda)$ are stable under taking vertical limits and horizontal translations.

\medskip

The first observation shows that the translation invariance of the Haar measure on $G$ and the definition of $\Hcal_{p}(\lambda)$  immediately imply that all mappings $D  \mapsto D^{\omega}$
constitute onto isometries on $\Hcal_{p}(\lambda)$.
\begin{Prop} \label{translationII} Let $(G, \beta)$ be a $\lambda$-Dirichlet group, $1\le p \le \infty$ and $D \in \Hcal_{p}(\lambda)$. Then for
$\omega \in G$ we have that $D\in \Hcal_{p}(\lambda)$ if and only if
$D^{\omega}\in \Hcal_{p}(\lambda)$, and in this case
and $\|D^{\omega}\|_{p}=\|D\|_{p}$.
\end{Prop}

The  following corollary applies the case $p=\infty$ to polynomials.

\begin{Coro} \label{isometrytrans2}  Let $(G, \beta)$ be a $\lambda$-Dirichlet group. Then for all $\omega\in G$ and $a_{1}, \ldots, a_{N} \in \C$
\begin{equation*}
\sup_{s\in  [Re>0]}\Big|  \sum_{n=1}^{N} a_{n}e^{-\lambda_{n}s}\Big| =\sup_{t\in \R} \Big|  \sum_{n=1}^{N} a_{n}e^{-t\lambda_{n}i}\Big|=\sup_{t\in \R} \Big|  \sum_{n=1}^{N} a_{n}\omega(\lambda_{n})e^{-t\lambda_{n}i}\Big|\,.
\end{equation*}
\end{Coro}

\begin{proof}
Applying Proposition \ref{translationII} to polynomials in $Pol_{\lambda}(G)$ (and 'restricting' them to $\R$) we obtain the second equality. The first equality  follows from a standard maximum modulus  principle.
\end{proof}

Let us prove the analog of  Proposition \ref{translationII}  for $\mathcal{D}_\infty(\lambda)$.

\begin{Prop} \label{rotation}
Let $(G, \beta)$ be a $\lambda$-Dirichlet group,   $D:=\sum a_{n} e^{-\lambda_{n}s} \in \mathcal{D}_{\infty}(\lambda)$,
and assume that $\mathcal{D}_{\infty}(\lambda)$ is complete.
Then for all
$\omega \in G$ we have that $D\in \mathcal{D}_\infty(\lambda)$ if and only if
$D^{\omega}\in \mathcal{D}_\infty(\lambda)$, and in this case
$\|D^{\omega}\|_{\infty}=\|D\|_{\infty}$.
\end{Prop}

\begin{proof}
Indeed,  by Proposition~\ref{typicalmeans}
the Dirichlet series $D=\sum a_{n} e^{-\lambda_{n}s} \in \mathcal{D}_{\infty}(\lambda)$ is approximated by typical means. Hence, given $\omega \in G$ and $\varepsilon >0$, by Corollary~\ref{isometrytrans2} the sequence  $$\left(\sum_{\lambda_{n}<x} a_{n}\left(1-\frac{\lambda_{n}}{x}\right)\omega(\lambda_{n})e^{-\varepsilon \lambda_{n}} e^{-\lambda_{n}s}\right)_{x\ge 0}$$
is a Cauchy sequence in $\mathcal{D}_\infty(\lambda)$. Then
 $D^{\omega}_{\varepsilon} \in \mathcal{D}_{\infty}(\lambda)$ and  $\|D^{\omega}_{\varepsilon}\|_{\infty}=\|D_{\varepsilon}\|_{\infty}$ for all $\omega \in G$ and $\varepsilon >0$. This
 immediately gives $D^{\omega} \in \mathcal{D}_{\infty}(\lambda)$  and $\|D^{\omega}\|_{\infty}=\|D\|_{\infty}$ for all $\omega \in G$.
\end{proof}

In the rest of this section we  justify  the chosen name 'vertical limit'.

\begin{Prop} Let $(G, \beta)$ be a $\lambda$-Dirichlet group, and $D=\sum a_{n}e^{-\lambda_{n}s}$ a $\lambda$-Dirichlet series with $\sigma_{a}(D)\le 0$. Then
\begin{itemize}
\item[(1)]
For every  $\omega \in G$ there is a sequence $(\tau_{k})_{k} \subset \R$ such that $D_{i\tau_{k}}$ converges to $D^{\omega}$ uniformly on $[Re>\varepsilon]$ for all $\varepsilon>0$.

\item[(2)]
 Assume conversely that for  $(\tau_{k})_{k} \subset \R$ the  vertical translations $D_{i\tau_k}$ converge
 uniformly on $[Re>\varepsilon]$ for every  $\varepsilon>0$
 to a holomorphic function $f$ on $[Re>0]$. Then there is $\omega \in G$ such that
    $f(s)= D^{\omega}(s)$
    for all $s \in [Re>0]$\,.
\end{itemize}
\end{Prop}
\begin{proof}
$(1)$ By Pontryagin's duality theorem we have $G=\widehat{(\widehat{G},d)}$ (as topological groups), where the right group carries the compact open topology. Hence in the following we interpret  $\omega \in G$ as a character on $\widehat{G}$. Since $\R$ is dense in $G$, by  definition of the compact open topology, for every $N \in \N$ there is $\tau_{N} \in \R$ such that
$$\max_{1\le j \le N} |e^{-i\lambda_{j}\tau_{N}}-\omega(\lambda_{j})|\le \frac{1}{N}.$$
Hence the sequence $(e^{-i\lambda_{j}\tau_{N}})_{j}$ converges to $(\omega(\lambda_{j}))_{j}$ in $\T^{\infty}$ whenever $N$ tends to $\infty$.
For each $N$ we define $D_{i\tau_{N}}(s):=\sum a_{j}e^{-i\lambda_{j}\tau_{N}} e^{-\lambda_{j}s}$ and take $\varepsilon>0$. Then  for all $s = \sigma+it \in [Re>\varepsilon]$ and $M,N \in \N$ we obtain
\begin{align*}
|D_{i\tau_{N}}(s)-D^{\omega}(s| & \le \sum_{n=1}^{\infty}|a_{n}| e^{-\lambda_{n}\varepsilon} |e^{-i\lambda_{n}\tau_{N}}-\omega(\lambda_{n})|
\\ &\le 2\sum_{n=M}^{\infty} |a_{n}|e^{-\lambda_{n}\varepsilon}  + \max_{1 \leq n \leq M} |a_{n}| e^{-\lambda_{n}\varepsilon}\sum_{n=1}^{M}  |e^{-i\lambda_{n}\tau_{N}}-\omega(\lambda_{n})|.
\end{align*}
Since $\sigma_{a}(D)\le 0$, we may choose $M$ and $N$ large enough to get the desired conclusion.
\\
$(2)$ We claim that there is $\omega \in G$ and a subsequence $(\tau_{n_{k}})_{k}$ such that $(e^{-i \tau_{n_{k}}\lambda_{j}})_{j}$ converges to $(\omega(\lambda_{j}))_{j}$ in $\T^{\infty}$ whenever  $k$ tends to $\infty$.
Then we may argue as in the proof of  $(1)$ to show that $D_{i\tau_{n_{k}}}$ converges to $D^{\omega}$ uniformly on $[Re>\varepsilon]$ for all $\varepsilon>0$, and consequently  $f=D^{\omega}$ on $[Re>0]$. For $m \in \N$
we define  $T_{m}:=\overline{\{\beta(\tau_{n}) \mid n \ge m\}}$, where the closure is taken in $G$. Then $\bigcap_{m \in A} T_{m}$ for all finite subsets $A \subset \N$, and since  $G$ is compact, we find some
$ \omega \in \bigcap_{m \in \N} T_{m} \ne \emptyset$. Consider now the following zero neighborhoods in $G$
 $$U_{N}:=\left\{ \eta \in G \mid \max_{1\le j \le N} |\eta(\lambda_{j})-1| \le \frac{1}{N} \right\}, ~ N \in \N.$$ So for $N=1$ there is $\tau_{k_{1}}$ such that $\beta(\tau_{k_{1}}) \in \omega +U_{1}$, since $\omega \in T_{1}$. Then for $N=2$ there is $\tau_{k_{2}}$, $k_{2}\ge k_{1}$, such that $\beta(\tau_{k_{2}}) \in \omega+ U_{2}$, since $\omega \in T_{k_{1}}$. Inductively we get a subsequence $(\tau_{n_{k}})_{k}$ such that
$$\max_{1\le j \le k} |e^{-i\tau_{n_{k}}\lambda_{j}}-\omega(\lambda_{j})| \le \frac{1}{k}$$ for all $k \in \N$. This proves the claim.
\end{proof}

\subsection{$\pmb{\mathcal{H}_{p}(\lambda)}$ equals $\pmb{\mathcal{H}^+_{p}(\lambda)}$ } \label{hp=hp+}

The next result shows that a   Dirichlet series  $D:=\sum a_{n} e^{-\lambda_{n}s}$ belongs to $ \Hcal_{p}(\lambda)$ if and only if all its translations $D_z, z \in [Re >0]$ do with  uniformly bounded norms.
For ordinary $\mathcal{H}_p$'s this was first proved in \cite{AntonioDefant} (see also \cite[Theorem 11.21]{Defant}).

\begin{Theo} \label{translationI} Let $1\le p < \infty$ and  $D:=\sum a_{n} e^{-\lambda_{n}s}$. Then $D \in \Hcal_{p}(\lambda)$ if and only if $D_{z}\in \Hcal_{p}(\lambda)$ for all $z \in [Re>0]$ and $\sup_{z \in [Re>0]} \|D_{z}\|_{p} < \infty$. Moreover, in this case $\sup_{z \in [Re>0]} \|D_{z}\|_{p} =\|D\|_{p}$.
\end{Theo}

We will give two important applications of this description of $\Hcal_{p}(\lambda)$ in the coming two
sections.

\medskip

For the  proof it is    convenient to define the Banach space $\mathcal{H}^+_{p}(\lambda)$  of all $\lambda$-Dirichlet series  such that
$\|D\|_{\mathcal{H}^+_{p}(\lambda)}:=\sup_{z \in [Re>0]} \|D_{z}\|_{p} < \infty$, and then the preceding equivalence in short tells  that $\mathcal{H}_{p}(\lambda)=\mathcal{H}^+_{p}(\lambda)$ holds isometrically.

\medskip

We show each of the implications of the preceding theorem separately.
The 'if part' is based on a lemma which relates horizontal translations of Dirichlet series  $D \in \Hcal_{p}(\lambda)$ (i.e. the Dirichlet coefficients of $D$ are Fourier coefficients of a function $f \in H_{p}^{\lambda}(G)$) with the convolution of $f$ with the Poisson measure $p_u$ on $G$
(see Lemma~\ref{measuresembedd}).

\begin{Lemm}\label{limituto0} Let $(G,\beta)$ be a $\lambda$-Dirichlet group and  $1\le p \le \infty$. Moreover,
let $f \in H_p^{\lambda}(G)$ and  $D \in \Hcal_{p}(\lambda)$ with $\Bcal(f)=D$. Then for each $u>0$, we have that $p_{u}*f \in H_p^{\lambda}(G)$, $D_u \in \Hcal_{p}(\lambda)$ with
 $\Bcal(p_{u}*f)=D_{u}$,  and $\lim_{u\to 0} D_{u}=D$ in $\Hcal_{p}(\lambda)$ (for $p=\infty$ with respect to the weak star topology
generated  by the duality of $L_\infty(G)$ and $L_1(G)$).
\end{Lemm}
\begin{proof}
Since $\widehat{p_{u}*f}(h_x)=\widehat{p_{u}}(h_x)\widehat{f}(h_x)=e^{-u|x|} \widehat{f}(h_x)$
for all $x \in \widehat{\beta}(\widehat{G})$, we have $p_{u}*f \in H_{p}^{\lambda}(G)$ with $\|p_{u}*f\|_{p}\le \|f\|_{p}$, and also  $\mathcal{B}(p_{u}*f)=D_{u}$. For the second statement notice that $\lim_{u\to 0} p_{u}*f=f$, where the limit for
$1 \leq p < \infty$ is taken in $L_p(G)$, and  for $p=\infty$ in the weak star topology (check for polynomials, then use density).
\end{proof}

\begin{proof}[Proof of the 'if part' of  Theorem~\ref{translationI}]
 Proposition \ref{translationII}  implies that $\|D_{it}\|_{p}=\|D\|_{p}$
 (choose $\omega = \beta(t)$). Then  together with Lemma \ref{limituto0} $$\sup_{z \in [Re>0]} \|D_{z}\|_{p}=\sup_{u>0} \|D_{u}\|_{p}=\|D\|_{p}\,. \qedhere
$$
\end{proof}

For the proof of the 'only if' part of  Theorem~\ref{translationI} we follow an idea which was invented in  \cite{AntonioDefant}.
Given a Banach space $X$,  we denote by $\mathcal{D}_{\infty}(\lambda,X)$ the space of all $\lambda$-Dirichlet series $D=\sum a_{n} e^{-\lambda_{n}s}$, which have coefficients  $a_{n} \in X$ and which converge and define a bounded (and then necessarily holomorphic) function on $[Re>0]$ with values in $X$. Together with the supremum norm this
gives a Banach space provided $\mathcal{D}_{\infty}(\lambda)$ is complete (see \cite{Schoolmann}).

\begin{Lemm} \label{embedding} Let $1\le p < \infty$, and assume that $\mathcal{D}_{\infty}(\lambda)$ is complete. Then the map
$$ \Phi \colon \Hcal_{p}(\lambda) \hookrightarrow \mathcal{D}_{\infty}(\lambda, \Hcal_{p}(\lambda)), ~~ \sum a_{n} e^{-\lambda_{n}s} \mapsto \sum \left(a_{n} e^{-\lambda_{n}s}\right)  e^{-\lambda_{n}z}$$
is an into isometry.
\end{Lemm}

In the ordinary case this was proved in \cite[Theorem 2.4]{AntonioDefant} (see also \cite[Section 11.3]{Defant}), and substituting $\T^{\infty}$ by any $\lambda$-Dirichlet group $(G.\beta)$ our proof follows the same strategy.

\begin{proof}
 We consider a $\lambda$-Dirichlet polynomial
$D = \sum_{n=1}^N a_n e^{-\lambda_n s}$. Then, by the 'if part' of Theorem~\ref{translationI},
 the map
$$F \colon [Re>0] \to \Hcal_{p}(\lambda), ~~ z \mapsto \sum_{n=1}^{N} a_{n} e^{-\lambda_{n}z} e^{-\lambda_{n}s}$$
is holomorphic  and $\|F\|_{\infty}=\|D\|_{p}$. Hence $\Phi$, defined as above
on polynomials only,
is isometric. Let $\Phi$ be the unique extension of this isometry to all of $\Hcal_{p}(\lambda)$
(Theorem~\ref{comp_pol}).  We claim that this is the map $\Phi$ from  the lemma. Indeed,
given $D \in \mathcal{H}_p(\lambda)$, let $(D^{N})_{N}$ be a sequence of Dirichlet polynomials which in $\mathcal{H}_p(\lambda)$ converges  to $D$. We  write $\Phi(D)=\sum b_{n} e^{-\lambda_{n}z}$, where $b_{n} \in \Hcal_{p}(\lambda)$. Then  the continuity of $\Phi$ implies $b_{n}=\lim_{N} a_{n}( \Phi(D^{N}))= \lim_{N} a_{n}(D^{N}) e^{-\lambda_{n} s}=a_{n}(D)e^{-\lambda_{n}s},$
the conclusion.
\end{proof}

Now we are prepared for the

\begin{proof}[Proof of the 'only if' part of  Theorem~\ref{translationI}] Let $D=\sum a_{n}e^{-\lambda_{n}s} \in  \mathcal{H}_{p}^{+}(\lambda)$, i.e. $D_{\varepsilon}\in \mathcal{H}_{p}(\lambda)$
for all $\varepsilon >0$ and
$\sup_{\varepsilon>0} \|D_\varepsilon\|_p =\|D\|_{\mathcal{H}_{p}^{+}(\lambda)}< \infty$. Hence  by Lemma \ref{embedding} the maps
  $$F_{\varepsilon}(D)(z):=\Phi(D_{\varepsilon})(z)= \sum \left(a_{n}e^{-\varepsilon\lambda_{n}}e^{-\lambda_{n}s}\right) e^{-\lambda_{n}z} \colon [Re>0] \to \mathcal{H}_{p}(\lambda) $$
   are holomorphic for all $\varepsilon >0$, and bounded uniformly in $\varepsilon$ by $\|D\|_{\mathcal{H}_{p}^{+}(\lambda)}$. Now the function $$F(z):=\sum a_{n}e^{-\lambda_{n}s} e^{-\lambda_{n}z} \colon [Re>0] \to \mathcal{H}_{p}(\lambda)$$ is holomorphic and bounded. Let $\varphi(z):=\frac{1+z}{1-z} \colon \D \to [Re>0]$ be the Cayley transformation (with inverse $\varphi^{-1}(s)=\frac{s-1}{s+1}$), and consider the function $f:=F\circ \varphi \colon \D \to \mathcal{H}_{p}(\lambda)$. Then $\lim_{\varepsilon\to 0}\varphi^{-1}(\varepsilon+it)=\frac{it-1}{it+1}=:w \in \T$ and as a matter of fact the so called Stolz region $S(\alpha, w):=\{ z \in \D \mid |z-w|<\alpha(1-|z|)\}$ for any $\alpha>1$ contains the set $\{\varphi^{-1}(\varepsilon+it) \mid 0<\varepsilon<\varepsilon_{0} \}$ for some $\varepsilon_{0}>0$  (see \cite[\S 11.4]{Defant}).  Since $\mathcal{H}_{p}(\lambda)$, being a closed subspace of some $L_{p}(G)$, has the analytic Radon-Nikodym property, the limits $\lim_{ S(\alpha,w)\ni z \to \frac{it-1}{it+1}}f(z)$ exists for almost all $t \in \R$. But then
$$\lim_{\varepsilon\to 0} F(\varepsilon+it)=\lim_{ S(\alpha,w)\ni z \to \frac{it-1}{it+1}}f(z)$$ exists in $\mathcal{H}_{p}(\lambda)$ for almost all $t \in \R$ and equals $D_{it}$ (since the Dirichlet coefficients of $F(\varepsilon+it)$ are given by $(a_{n}(D)e^{-\lambda_{n}(\varepsilon+it)})$). Now Proposition \ref{translationII} (applied to some admissible $\omega:=\beta(t)$) gives $D\in \mathcal{H}_{p}(\lambda)$.
\end{proof}

\smallskip

\subsection{$\pmb{\mathcal{D}_{\infty}(\lambda)}$ equals $\pmb{\mathcal{H}_{\infty}(\lambda)}$ } \label{fourier}

As already mentioned, there are now two  `$H_\infty$-spaces of $\lambda$-Dirichlet series' around, namely $\mathcal{D}_{\infty}(\lambda)$ and $\Hcal_{\infty}(\lambda)$.
 Note that by \cite{DefantSchoolmann2}  there are  frequencies $\lambda$
such that $\mathcal{D}_{\infty}(\lambda)$ is not complete, hence in these cases  $\mathcal{D}_{\infty}(\lambda)\neq\mathcal{H}_{\infty}(\lambda)$. In this subsection we show that
both  spaces  coincide isometrically for $\lambda$'s satisfying $\mathcal{D}^{ext}_{\infty}(\lambda)=\mathcal{D}_{\infty}(\lambda)$ and $L(\lambda)<\infty$. See
Theorem \ref{conditions2} as well as
Theorem \ref{conditions} for 'testable' sufficient conditions on  this.

\begin{Theo} \label{Dstar} Let $(G,\beta)$ be a $\lambda$-Dirichlet group, and $\lambda$ a frequency. Then there is an injective contraction map
\begin{equation} \label{Psi}
\Psi \colon \mathcal{D}^{ext}_{\infty}(\lambda) \to H^{\lambda}_{\infty}(G),~~ D \mapsto f
\end{equation}
such that $a_{n}(D)=\widehat{f}(h_{\lambda_{n}})$ for all $n \in \N$.
\end{Theo}

 An interesting  by-product of Theorem~\ref{Dstar} is that  under no additional further assumptions on the frequency $\lambda$, each Dirichlet series $D \in \mathcal{D}_{\infty}(\lambda)$ has Dirichlet coefficients
which are Fourier coefficients of an $L_\infty$-function on a compact abelian group
(e.g. the Bohr compactification of $\mathbb{R}$).
Hence the  following corollary is an immediate consequence of Parseval's equality.

\begin{Coro}
\label{H2}
Let $D \in \mathcal{D}^{ext}_{\infty}(\lambda)$ with Dirichlet coefficients $(a_{n})_{n}$. Then
$$\left(\sum_{n=1}^{\infty} |a_{n}|^{2}\right)^{\frac{1}{2}} \le \|D\|_{\infty}.$$
\end{Coro}
In view of Theorem~\ref{conditions2} the following results collects a couple of sufficient conditions under which
$\mathcal{D}_\infty(\lambda)$ and $\mathcal{H}_\infty(\lambda)$ even coincide isometrically.

\begin{Theo} \label{mainresultinfty}  Let $(G,\beta)$ be a  $\lambda$-Dirichlet group, and assume that $L(\lambda)<\infty$ and that $\mathcal{D}^{ext}_{\infty}(\lambda)=\mathcal{D}_{\infty}(\lambda)$. Then the mapping $\Psi$ from Theorem \ref{Dstar} is an onto isometry and its inverse is given by the Bohr map
\begin{equation} \label{borro}
\Bcal\colon H_{\infty}^{\lambda}(G) \to \mathcal{D}_{\infty}(\lambda),~~ f \mapsto \sum \widehat{f}(h_{\lambda_{n}}) e^{-\lambda_{n}s}.
\end{equation}
In particular, $\mathcal{D}_{\infty}(\lambda)=\mathcal{H}_{\infty}(\lambda)$ isometrically.
\end{Theo}

For polynomials the proof of   the desired   norm equality  is rather easy. Since   $\beta(\R)$ is dense in  $G$,
a standard distinguished maximum modulus principle gives the following

\begin{Lemm} \label{isompolyinfty} Let $(G,\beta)$ be a Dirichlet group. Then for all $a_{1}, \ldots a_{N} \in \C$ and $ \lambda_1, \ldots, \lambda_N \in \R_{\ge 0}$
$$
\sup_{ \text{Re}\, s > 0} \bigg| \sum_{n=1}^{N} a_{n}e^{-s\lambda_{n}}\bigg|
=
\sup_{t \in \R} \bigg| \sum_{n=1}^{N} a_{n}e^{-it\lambda_{n}}\bigg| = \sup_{\omega \in G} \bigg|\sum_{n=1}^{N} a_{n}\omega(\lambda_{n}) \bigg|.$$
\end{Lemm}

This gives the desired isometry between $(Pol(\lambda), \|\cdot\|_{\infty})$ and $(Pol_{\lambda}(G),\|\cdot\|_{\infty})$. We have to prove that statement beyond polynomials.

\medskip

\begin{proof}[Proof of Theorem \ref{Dstar}]
Let $D=\sum a_{n}e^{-\lambda_{n}s} \in \mathcal{D}^{ext}_{\infty}(\lambda)$ with extension $F$, and
consider
for $x, \sigma >0$ the polynomials
$$R_x^\sigma(\omega) := R_{x}(D_\sigma)(\omega)= \sum_{\lambda_{n}<x} a_{n} \left(1-\frac{
\lambda_{n}}{x}\right) e^{-\sigma\lambda_{n}} \omega(\lambda_{n})$$
 from Theorem \ref{typicalmeans} which are associated to $D_\sigma$.
Then by Lemma \ref{isompolyinfty} the net $(R_x^\sigma)_x$  for each $\sigma$ forms a  Cauchy net in $H_{\infty}^{\lambda}(G)$ with limit, say $f_{\sigma}$. Moreover,
for all $n, \sigma$
$$\widehat{f_{\sigma}}(h_{\lambda_{n}})=\lim_{x \to \infty} \widehat{R_{x}^{\sigma}}(h_{\lambda_{n}})=a_{n}e^{-\sigma\lambda_{n}}\,,$$
and
$$\|f_{\sigma}\|_{\infty}=\lim_{x \to \infty} \|P^{\sigma}_{x}\|_{\infty}=\|D_{\sigma}\|_{\infty}\le \|F\|_{\infty}.$$
Now recall that the unit ball of $L_{\infty}(G)$ together with its weak star topology
is metrizable and compact (by the Alaoglu-Bourbarki theorem and the fact that $L_{1}(G)$
is separable). Hence, $(f_{1/n})_n$ has a weak star convergent subsequence $(f_{1/n_k})_k$ with limit
$f \in L_{\infty}(G)$, $\|f\|_\infty \leq \|D\|_{\infty}$. Then for each $n$
$$a_n = \lim_k a_n e^{-\frac{1}{n_k} \lambda_n} = \lim_k \int_{\overline{\R}}f_{\frac{1}{n_k}}(w) \overline{h_{\lambda_n}(w)}dw = \widehat{f}(h_{\lambda_n})\,,$$
and moreover for each $x \notin (\lambda_{n})_{n}$
$$\widehat{f}(h_{x})=\lim_{k}\int_{\overline{\R}}f_{\frac{1}{n_k}}(w) \overline{h_{x}}(w)dw=0,$$
since $f_{\frac{1}{n_k}} \in H_{\infty}^{\lambda}(G)$. Hence $f \in H_{\infty}^\lambda(G)$ and $\Psi(D) = f$\,.
\end{proof}

The proof of Theorem~\ref{mainresultinfty} needs more preparation, and will be given after the following result inspired by the work of Helson from \cite{Helson}.

\begin{Lemm}
Let $(G,\beta)$ be a $\lambda$-Dirichlet group and  $L(\lambda)<\infty$.
For  $f \in H_{\infty}^{\lambda}(G)$  consider the Dirichlet series
$D=\sum a_{n} e^{-\lambda_{n}s}$, where $a_{n}:=\widehat{f}(h_{\lambda_{n}})$  for $n \in \mathbb{N}$.
Then $D^{\omega}\in \mathcal{D}^{ext}_{\infty}(\lambda)$ for  almost all $\omega \in G$, and moreover
\[
f_{\omega}*P_{u}(t)=D^{\omega}(u+it)
\]
for  almost all $\omega \in G$ and for all $u+it\in [Re> \max\{0,\sigma_{c}(D^{\omega})\}]$\,.
\end{Lemm}
\begin{proof} Since $L(\lambda)<\infty$ and $(a_{n})_{n}$ is bounded, we have that $\sigma_{a}(D^{\omega})\le L(\lambda)<\infty$ for all $\omega \in G$. We already know from Proposition \ref{holomorphy} that
there is a null set $N \subset G$ such for all $\omega \notin N$
the function
$$F_\omega\colon [Re>0] \to \C\,,\,\, F_{\omega}(u+it):=(f_{\omega}*P_{u})(t)$$
is holomorphic and bounded by $\|f\|_{\infty}$. On the other side $D^{\omega}$ defines a holomorphic function on $[Re>\sigma_{c}(D^{\omega})]$ for all $\omega \in G$. So it suffices to verify that there is some null set $M \subset G$  which contains $N$ and is such that $F_{\omega}=D^{\omega}$ on $[Re=u]$ for  all $u>\max\{0,\sigma_{c}(D^{\omega})\}$
 and   $\omega \notin M$.  Let $u:=L(\lambda)+1$ (note that $L(\lambda)=0$ is possible) and as in the proof of Proposition \ref{holomorphy} we find a sequence $(N_k)$ in $\mathbb{N}$ and a null set $N_u \subset G$ such that for all
 $\omega \notin N \cup N_u$ we have  $\lim_{k\to \infty}\frac{f^{N_{k}}_{\omega}*P_{u}}{u+i\cdot}=\frac{f_{\omega}*P_{u}}{u+i\cdot}$
   in $L_2(\mathbb{R})$. Hence for every $\omega \notin N \cup N_u$   there is a subsequence $(N_{k_{j}})_{j}$ such that for almost all $t \in \R$
   \[
   D^{\omega}(u+it)=\lim_{j \to \infty} \sum_{n=1}^{N_{k_{j}}} a_{n}\omega(\lambda_{n})e^{-u\lambda_{n}}e^{-\lambda_{n}it}=f_{\omega}*P_{u}(t) =F_\omega(u+it)\,.
   \]
But for $\omega \notin N \cup N_u$ both functions $D^{\omega}$ and  $F_{\omega}$ are  continuous on $[Re=u]$, and so with $M:=N\cup N_{u}$ we have $D^{\omega}=F_{\omega}$ on $[Re=L(\lambda)+1]$ for all $\omega \notin M$. Finally, the identity principle  from complex analysis  implies that $D^{\omega}=F_{\omega}$ on $[Re>\max\{0,\sigma_{c}(D^{\omega})\}]$ for all $\omega \notin M$.
\end{proof}

We obtain as an immediate consequence the following variant
of an important theorem on general $\lambda$-Dirichlet series in $\mathcal{H}_2(\lambda)$
due to Helson  from \cite{Helson3}.

\begin{Theo} \label{helson} Assume that $L(\lambda)<\infty$ and  $\mathcal{D}^{ext}_{\infty}(\lambda)=\mathcal{D}_{\infty}(\lambda)$, and let   $D \in \mathcal{H}_{\infty}(\lambda)$. Then for all $\lambda$-Dirichlet groups $(G, \beta)$ and for almost all $\omega \in G$ the vertical limits $D^\omega$   converge on $[Re >0]$.

 Moreover, if $f \in H_{\infty}^{\lambda}(G)$ in view of Bohr's map $\Bcal$ corresponds to $D$, then
for almost all $\omega \in G$ and all $u+it \in [Re>0]$
 \begin{equation} \label{integralformula}
P_{u}*f_{\omega}(t)= D^\omega(u+it).
\end{equation}

\end{Theo}
In \cite{DefantSchoolmann3} we extend Theorem \ref{helson} to $\mathcal{H}_{p}(\lambda)$, $1\le p < \infty$, adding the relevant maximal inequality.

\medskip

Finally, we close the circle and prove

\begin{proof}[Proof of Theorem~\ref{mainresultinfty}]  Take $f \in H_{\infty}^{\lambda}(G)$, and let
$D\in \mathcal{H}_{\infty}(\lambda)$ be the associated  $\lambda$-Dirichlet series  under the Bohr map $\mathcal{B}$.
From Theorem \ref{helson} we know that there is some  $\omega \in G$ such that $D_{\omega} \in \mathcal{D}_{\infty}(\lambda)$ with $\|D_{\omega}\|_{\infty}\le \|f\|_{\infty}$. Then, applying Proposition~\ref{rotation} with $\omega^{-1}$, we obtain $D \in \mathcal{D}_{\infty}(\lambda)$ and $\|D\|_{\infty}\le \|f\|_{\infty}$.
 \end{proof}
\subsection{Schauder basis} \label{Schauder}
It is known that $(n^{-s})$ is a Schauder basis for $\Hcal_{p}(\log(n))$ in the range $1<p<\infty$ (see \cite{Saksman}). This result extends to arbitrary frequencies $\lambda$.

\begin{Theo} \label{schauderbasisinHp}
Let $1 < p < \infty$ and $\lambda$ be a frequency. Then the monomials $(e^{-\lambda_{n} s})$ form a Schauder basis for $\Hcal_{p}(\lambda)$.
\end{Theo}
\begin{proof} Let $(G, \beta)$ be any $\lambda$-Dirichlet group. We then claim that $(h_{\lambda_{n}})$ is a Schauder basis for $H^{\lambda}_{p}(G)$. By Proposition~\ref{densitytrigonometric} it suffices to verify, that there is a constant $C\ge 1$, such that for all $m>n$ and for all $a_{1}, \dotsc, a_{m} \in \C$
\begin{equation*}
\Big\|\sum_{n=1}^{n} a_{k} h_{\lambda_{k}}\Big\|_{p} \le C \Big\|\sum_{n=1}^{m} a_{k} h_{\lambda_{k}}\Big\|_{p}.
\end{equation*}
By Proposition~\ref{hx} there is a natural order on $\widehat{G}$: We call a character
$h_{x}, x \in \widehat{\beta}(\widehat{G}),$ positive if $x \ge 0$. Then the 'Riesz projection'
 \begin{equation*}
\Phi\left(\sum a_{k} h_{x_{k}}\right):=\sum_{x_{k}\ge 0} a_{k} h_{x_{k}}
\end{equation*}
 is bounded on $Pol(G) \subset L_p(G)$ with norm $C=C(p)>0$ (see \cite[Theorem 8.7.2.]{Rudin62}, here connectedness of $G$ is needed), and hence  for  $m> n$ and each $f:=\sum_{k=1}^{m} a_{k} h_{\lambda_{k}} \in Pol_{\lambda}(G)$  we as desired have
\begin{align*}
\Big\|\sum_{k=1}^{n}a_{k}h_{\lambda_{k}}\Big\|_{p}& \le \Big\|\sum_{k=1}^{m}a_{k}h_{\lambda_{k}}\Big\|_{p}+\Big\|\sum_{k=n+1}^{m}a_{k}h_{\lambda_{k}}\Big\|_{p} \\ &=\Big\|\sum_{k=1}^{m}a_{k}h_{\lambda_{k}}\Big\|_{p}+\Big\|h_{\lambda_{n+1}}\Phi\left(h_{-\lambda_{n+1}} \sum_{k=1}^{m} a_{k} h_{\lambda_{k}}\right)\Big\|_{p}\\ &\le (1+C)\Big\|\sum_{k=1}^{m} a_{k} h_{\lambda_{k}}\Big\|_{p}.\qedhere
\end{align*}
\end{proof}
An equivalent formulation of  Theorem \ref{schauderbasisinHp}  is that
for $1<p<\infty$
all projections
\[
\pi^N_p: \mathcal{H}_p(\lambda) \to \mathcal{H}_p(\lambda),\,\,
\sum a_n e^{-\lambda_n s} \mapsto \sum_{n=1}^N a_n e^{-\lambda_n s}
\]
 are uniformly bounded. But for the border cases  $p=1$ and $p=\infty$ this in general is false
 (e.g., for the frequencies  $\lambda = (\log n)$  or $\lambda= (n)$). We give an upper bound for the growth of the partial sum operators in $\Hcal_{1}(\lambda)$; thereby the quality of the upper bounds depends on the quality of $\lambda$.
\begin{Prop} \label{basisconstant} Let $\lambda$ be a frequency
and $p=1$ or $\infty$. Assuming $(BC)$ for $\lambda$ there is a constant $C=C(\lambda)$ such that for all $N$
\[
\|\pi^N_p: \mathcal{H}_p(\lambda) \to \mathcal{H}_p(\lambda)\|\leq C \lambda_{N}\,,
\]
and, assuming $(LC)$ and $L(\lambda)<\infty$, for every $\delta$ there is a constant $C_{1}=C_{1}(\delta,\lambda)$ such that  for all $N$
\[
\|\pi^N_p: \mathcal{H}_p(\lambda) \to \mathcal{H}_p(\lambda)\| \le C_{1}e^{\delta\lambda_{N}}\,.
\]
\end{Prop}
The case $p=\infty$  follows from a careful analysis of results due to Bohr \cite{Bohr} and Landau \cite{Landau}, and
this was done in  \cite[\S 3]{DefantSchoolmann2}. The case $p=1$,
in a sense, reduces to the case $p=\infty$, and in the following we sketch this argument.

\begin{Rema} \label{nibe}
Given a Banach space $X$, a straightforward  Hahn-Banach  argument shows
that for $p=\infty$  Proposition~\ref{basisconstant} transfers to $\mathcal{D}_\infty(\lambda, X)$, i.e. under the assumptions  of (BC) or [(LC) and $L(\lambda)<\infty$],  the projection on  $\mathcal{D}_\infty(\lambda, X)$ which assigns to each $D$ its  $N$th partial sum, is  bounded by a universal  constant times $\lambda_{N}$ and $e^{\delta\lambda_{N}}$, respectively.
\end{Rema}

\begin{proof}[Proof of Proposition \ref{basisconstant}]
Together with Lemma~\ref{embedding} and Remark~\ref{nibe} we obtain assuming $(BC)$ that for all
$D \in \mathcal{H}_p(\lambda)$
$$\left\| \sum_{n=1}^{N} a_{n}e^{-\lambda_{n}s} \right\|_{1} =
\sup_{\text{Re} z >0}
\left\| \sum_{n=1}^{N} \left(a_{n}e^{-\lambda_{n}\cdot}\right) e^{-\lambda_{n}z} \right\|_{\infty}\le C\lambda_{N}  \|D\|_{1}\,.$$
Under the  assumption[(LC) and $L(\lambda)<\infty$] the result follows similarly.
\end{proof}

\smallskip

\subsection{Montel  theorem}
Bayart proved in \cite[Lemma 18]{Bayart} that for every bounded sequence $(D^{N})_{N} \subset \mathcal{D}_{\infty}((\log n))$ there is a subsequence $(D^{N_{k}})_{k}$ and some $D \in \mathcal{D}_{\infty}((\log n))$ such that $(D^{N_{k}})_{k}$ converges to $D$ on $[Re> \varepsilon]$ for all $\varepsilon>0$. In \cite[\S 4.4]{DefantSchoolmann2} we extend this Montel type theorem  to $\mathcal{D}_{\infty}(\lambda)$. Here, using Lemma~\ref{embedding}, we transport this result to $\Hcal_{p}(\lambda)$.

\begin{Theo}\label{montel} Assume that  $\lambda$ satisfies $L(\lambda)=0$, or [(LC) and $L(\lambda)<\infty$], or  is  $\mathbb{Q}$-linearly independent. Given $1\le p \le \infty$, let $(D^{N})_{N} \subset \Hcal_{p}(\lambda)$ be a bounded sequence. Then there is a subsequence $(D^{N_{k}})_{k}$ and some $D \in \Hcal_{p}(\lambda)$ such that the translations
$D^{N_{k}}_{\varepsilon}$ converge to $D$ in $\Hcal_{p}(\lambda)$ for all $\varepsilon>0$.

\end{Theo}

\begin{proof} We give a proof under the assumption that $\lambda$ satisfies $L(\lambda)<\infty$ and $(LC)$. Let us write $D^{N}(s)=\sum a_{n}^{N}e^{-\lambda_{n}s}$. Then by Lemma \ref{embedding} the Dirichlet series $\Phi(D^{N}) \in \mathcal{D}_{\infty}(\lambda,\mathcal{H}_{p}(\lambda))$ have the coefficients $a^{N}_{n}e^{-\lambda_{n}s} \in \mathcal{H}_{p}(\lambda)$ and so $|a_{n}^{N}|=\|a_{n}^{N}e^{-\lambda_{n}s}\|_{p}\le \|\Phi(D^N)\|_{\infty}\le \sup_{M\in \N} \|D^{M}\|_{p}:=C< \infty$ for all $n, N$. Hence by a diagonal process there is a subsequence $(N_{k})$ such the limits $\lim_{k\to \infty} a_{n}^{N_{k}}=:a_{n}$ exist for all $n $. We define $D:=\sum a_{n}e^{-\lambda_{n}s}$ and claim that $D \in \mathcal{H}_{p}(\lambda)$ and that $\lim_{k \to \infty} D_{\varepsilon}^{N_{k}}=D_{\varepsilon}$ in $\mathcal{H}_{p}(\lambda)$  for all $\varepsilon>0$.
By Lemma \ref{embedding} and Theorem~\ref{schauderbasisinHp} ($1 < p < \infty$) as well as Proposition \ref{basisconstant} ($p=1$ or $\infty$) for all $\delta>0$ there is $C_{1}=C_{1}(\delta, \lambda)$ such that for all $n,k$
\[
\sup_{z \in [Re>0]} \left\|\sum_{j=1}^{n} a_{j}^{N_{k}}e^{-\lambda_{j}s} e^{-\lambda_{j}z}\right\|_{p}
= \left\|\sum_{j=1}^{n} a_{j}^{N_{k}}e^{-\lambda_{j}s}\right\|_{p}
\le C_{1} e^{\delta\lambda_{n}}\|D^{N_{k}}\|_{p}\le C_{1} e^{\delta \lambda_{n}}C.
\]
Then the Bohr-Cahen formula for sequences of Dirichlet series from \cite[Proposition 2.4.]{DefantSchoolmann2} (for vector valued Dirichlet series
the formula follows as in  the scalar case  replacing the absolute value by the norm) implies that
 $E(z):=\sum a_{n}e^{-\lambda_{n} s} e^{-\lambda_{n}z}$ converges on $[Re>0]$ and that $\Phi(D^{N_{k}})$ converges to $E$ uniformly on $[Re>\varepsilon]$ for all $\varepsilon>0$. Hence $E\in \mathcal{D}_{\infty}(\lambda, \mathcal{H}_{p}(\lambda))$. Since $\Phi(D^{N_{k}}_{\varepsilon})(z)=\Phi(D^{N_{k}})(\varepsilon+z)$ on $[Re>0]$, this implies, again using Lemma \ref{embedding}, that $(D^{N_{k}}_{\varepsilon})_{k}$ is Cauchy in $\mathcal{H}_{p}(\lambda)$ with limit $D_{\varepsilon} \in \mathcal{H}_{p}(\lambda)$ and $\|D_{\varepsilon}\|_{p}\le C$ for all $\varepsilon>0$.
Hence by Theorem \ref{translationI} we have as desired that
$D \in \mathcal{H}_{p}(\lambda)$ and  $\lim_{k \to \infty} D_{\varepsilon}^{N_{k}}=D_{\varepsilon}$ in $\mathcal{H}_{p}(\lambda)$  for all $\varepsilon>0$. The remaining cases ($\lambda$ satisfies $L(\lambda)=0$ or   is  $\mathbb{Q}$-linearly independent) follow with the same strategy using the corresponding quantitative version of Bohr's theorem from \cite{DefantSchoolmann2}.
\end{proof}

As an immediate consequence  we state the counterpart in $H_{p}^{\lambda}(G)$.
\begin{Coro}
\label{schonwieder}
Assume that $\lambda$ satisfies  the assumptions of Theorem \ref{montel}, and $1 \leq p \leq \infty$. Then for  every bounded sequence $(f^{N})_{N} \subset H_{p}^{\lambda}(G)$ there are  a subsequence $(f^{N_{k}})_{k}$ and a function $f \in H_{p}^{\lambda}(G)$ such that $p_{\sigma}*f^{N_{k}}$ converges to $p_{\sigma}*f$ for all $\sigma >0$.
\end{Coro}

\smallskip
\subsection{$N$th Abschnitte}

As already mentioned the 'ordinary' $\Hcal_{\infty}$ isometrically  equals $H_{\infty}(B_{c_{0}})$,
identifying Dirichlet and monomial coefficients. A cruicial argument in the proof of this result (see \cite[\S 2.3]{Defant}) is that a continuous function $f \colon B_{c_{0}} \to \C$ belongs to $H_{\infty}(B_{c_{0}})$ if and only if all restriction maps $f_{N} \colon \D^{N}\to \C$ belong to $H_{\infty}(\D^{N})$ with $\sup_{N} \|f_{N}\|_{\infty}<\infty$.

Formulated for  ordinary Dirichlet series this says that  $D=\sum a_n n^{-s}\in \Hcal_{\infty}$ if and only if
all of its  so-called $N$-th abschnitte
$
D|_{N} = \sum a_{n} n^{-s}\,,
$
where the sum is taken over those $n$ which have only the first $N$ primes as divisors,
belong to  $\Hcal_{\infty}$ with uniformly bounded norms.
   In more vague terms,  $D \in \mathcal{H}_\infty$ if and only if all its finite dimensional blocks are
    in $\mathcal{H}_\infty$ with uniformly bounded norms.

   This phenomenon is also true for general Dirichlet series. If $\lambda$ is a frequency with decomposition $(R,B)$, then the $N$-th abschnitt $D|_{N}$ of a $\lambda$-Dirichlet series $D$ is $\sum a_{n}(D) e^{-\lambda_{n}s}$, where $a_{n}(D)\ne 0$ implies that $\lambda_{n}\in \operatorname{span}_{\mathbb{Q}}(b_{1},\ldots, b_{N}).$
Looking at  the identification $\mathcal{H}_{p}(\lambda)=H_{p}^{R}(\widehat{\mathbb{Q}_d}^{\infty}), D \mapsto F$,
given in  Theorem~\ref{summary},
this restriction $D|_N$  simply corresponds to
\begin{equation*} \label{projectiontoN}
F|_{N}\colon \widehat{\mathbb{Q}_d}^{N} \to \C\,,\,\,\, F|_{N}(\omega):=\int_{\widehat{\mathbb{Q}_d}^{\infty}} F(\omega, \eta) dm(\eta)\,,
\end{equation*}
the  `restriction' of $F$  to the first $N$ variables. Then an application of H\"{o}lder's inequality gives $\|F|_{N}\|_{p}\le \|F\|_{p}$, and so (again by Theorem~\ref{summary}) $\|D|_{N}\|_{p}\le \|D\|_{p}$.

\begin{Rema}\label{proj} Let $\lambda=(R,B)$, $1\le p\le \infty$ and $D \in \mathcal{H}_{p}(\lambda)$. Then $D|_{N} \in \mathcal{H}_{p}(\lambda)$ with $\|D|_{N}\|_{p}\le \|D\|_{p}$ for all $N\in \N$.
\end{Rema}

\begin{Theo} \label{Nabschnitt}
Assume that  $1\le p \le \infty$,  $\lambda=(R,B)$ is a frequency satisfying one of the conditions of Theorem \ref{montel}, and  $D$  a  $\lambda$-Dirichlet series. Then $D \in \Hcal_{p}(\lambda)$ if and only if  $D|_{N} \in \Hcal_{p}(\lambda)$ for all $N$ and $\sup_{N} \|D|_{N}\|_{p}<\infty$. Moreover, in this case $\|D\|_{p}=\sup \|D|_{N}\|_{p}.$
\end{Theo}
Together with Theorem \ref{mainresultinfty} we obtain as an immediate consequence the following
\begin{Coro} Assume that $\lambda=(R,B)$ satisfies one of the conditions of Theorem \ref{montel}, and $D$ is a $\lambda$-Dirichlet series. Then $D \in \mathcal{D}_{\infty}(\lambda)$ if and only if  $D|_{N} \in \mathcal{D}_{\infty}(\lambda)$ for all $N $ and $\sup_{N} \|D|_{N}\|_{\infty}<\infty$. Moreover, in this case $\|D\|_{\infty}=\sup_{N \in \N} \|D|_{N}\|_{\infty}.$
\end{Coro}

In view of the definition of the $\mathcal{H}_p(\lambda)$'s from \ref{zugcrash} we may reformulate Theorem~\ref{Nabschnitt} in terms of Hardy spaces on
$\lambda$-Dirichlet groups.

\begin{Theo}\label{Nabschnitt2}

Let $\lambda=(R,B)$ be a frequency satisfying one of the condition of Theorem \ref{montel},  $(G,\beta)$  a $\lambda$-Dirichlet group, and $1 \leq p \leq \infty$.
  Then for every sequence $(a_{n})_{n}$ in $\C$ the following are equivalent:
\begin{enumerate}
\item[(1)]$\exists ~f \in H_{p}^{\lambda}(G)\,\, \forall ~ n \in \N:~ \widehat{f}(h_{\lambda_{n}})=a_{n}$
\item[(2)] $\forall N \in \N ~\exists f_{N} \in H_{p}^{\lambda}(G) :$ $$\widehat{f_{N}}(h_{\lambda_{n}})=
    \begin{cases} a_{n} & \text{if }\lambda_{n} \in \operatorname{span}_{\mathbb{Q}}(b_{1}, \ldots, b_{N}), \\ 0
    & \text{else}
    \end{cases} $$
and $\sup_{N \in \N} \|f_{N}\|_{p}< \infty$.
\end{enumerate}
Moreover, in this case $\|f\|_{p}=\sup_{N \in \N} \|f_{N}\|_{p}$.
\end{Theo}
\begin{proof}
Although the implication  $(1) \Rightarrow (2)$ already follows from Remark \ref{proj}, we like to give another argument of different flavour. For each $N$ we consider the subgroup $$U_{N}:=\operatorname{span}_{\mathbb{Q}} (b_{1}, \ldots, b_{N})\cap \widehat{G} \subset \widehat{G}\,,$$
and take, according to Lemma~\ref{idempotent}, some  $\mu_{N} \in M(G)$ with $\|\mu_{N}\|=1$ such that $\widehat{\mu_{N}}=\chi_{U_{N}}$. Then the functions $f_{N}:=f*\mu_{N}\in H_{p}^{\lambda}(G)$ fulfill the claim.

It remains to prove that $(2) \Rightarrow (1)$: Let us start with the case $p=1$. Denote by  $D^{N}$  those  Dirichlet series which correspond  to the functions $f_{N}$. Then by Theorem \ref{montel} there is $D \in \Hcal_{1}(\lambda)$ and a subsequence $(D^{N_{k}})$ such that $D^{N_{k}}_{\varepsilon} \to D_{\varepsilon}$
in $\Hcal_{1}(\lambda)$
for all $\varepsilon>0$. Then  $f:=\Bcal(D)$ has the right Fourier coefficients and by Theorem~\ref{translationI}
 \begin{align*}
  \|f\|_{1}
  =\|D\|_{1}
  &
  =\sup_{\varepsilon >0} \|D_{\varepsilon}\|_{1}
  \le
\sup_{\varepsilon >0}\sup_{N} \|D^{N}_\varepsilon\|_{1}
\\&
=
\sup_{N}\sup_{\varepsilon >0} \|D^{N}_\varepsilon\|_{1}
=
\sup_{N}\|D^{N}\|_{1}
=\sup_{N } \|f_{N}\|_{1} .
\end{align*}
In the case $1<p\le \infty$ the sequence $(f_{N})_{N}$ is   bounded,
 and hence weakly bounded  in $L_{p'}(G)$. Then by the Alaoglu-Bourbaki theorem  there is $f \in H_{p}^{\lambda}(G)$ with the right Fourier coefficients and $\|f\|_{p}\le \sup_{N } \|f_{N}\|_{p}$.
\end{proof}

\subsection{Brothers Riesz  theorem} \label{brotherRieszchapter}

The aim of our last section  is to discuss the following brothers Riesz type theorem.

\smallskip

\begin{Theo} \label{brotherRiesz}
Let $\lambda$ be a frequency and  $(G,\beta)$  a $\lambda$-Dirichlet group. Then the Bohr map
$$\Bcal \colon H_{1}^{\lambda}(G) \to M_{\lambda}(G), ~~f \mapsto f ~dm$$
 is an onto isometry preserving the Fourier coefficients.
\end{Theo}

Note that
 for the frequency $\lambda=(n)$  this result is the classical brothers Riesz theorem on $\mathbb{T}$, and
for the ordinary case $\lambda= (\log n)$ it means that the Hardy space $H_{1}(\T^{\infty})$ isometrically equals  the space $M_{+}(\T^{\infty})$ of all bounded, regular and analytic Borel measures on $\T^{\infty}$ (due to Helson and  Lowdenslager from \cite{HelsonLowdenslager}, see also \cite[\S 13.1]{Defant}).

The proof of Theorem \ref{brotherRiesz} is an immediate consequence of  \cite[Theorem 4]{Doss}, which in terms of Dirichlet groups states that every  $\mu \in M(G)$
is absolutely continuous
with respect to the Haar measure on $G$, whenever for every $\delta>0$ the set $S_{\delta}(\mu):=\{ h_{x} \mid x<\delta \text{ and } \widehat{\mu}(h_{x})\ne 0 \}$ is finite. Clearly, any $\mu \in M_{\lambda}(G)$
satisfies this condition. \\

We now like to give a proof of Theorem~\ref{brotherRiesz} within the theory of Dirichlet series by applying the Montel Theorem \ref{montel}. But  then we of course have to  put the additional assumptions on $\lambda$
needed for that theorem. Moreover, we need the existence of so called `local units', which follows  as an immediate consequence of  \cite[Theorem 2.6.8]{Rudin62}.
\begin{Lemm} \label{localunits} Let $\varepsilon>0$ and $G$ a $\lambda$-Dirichlet group. Then there is a sequence of polynomials $(k_{N})_{N} \subset Pol(G)$ such that $\widehat{k_{N}}(h_{\lambda_{j}})=1$ for all $N \in \mathbb{N}$ and all $j=1,\ldots, N$, and moreover
$\sup_{N} \|k_{N}\|_{L^{1}(G)}\le 1+\varepsilon$.
\end{Lemm}
Actually \cite[Theorem 4]{Doss} is an easy consequence of the generalized version of the brothers Riesz theorem by Helson and Lowdenslager (see \cite[\S 8.2.3.]{Rudin62}) and Lemma \ref{localunits}. So in this sense the Montel Theorem \ref{montel} replaces the former result in the proof of Theorem \ref{brotherRiesz} (but then for a smaller class of $\lambda$'s).

\begin{proof}[Proof of Theorem \ref{brotherRiesz} under the  assumptions on $\lambda$
needed in Theorem \ref{montel}]  Take some $\mu \in M_{\lambda}(G)$ and  $\varepsilon>0$. We choose $(k_{N})_{N}$ from Lemma \ref{localunits}. Then $(\mu *k_{N})_{N}$ is bounded in $H_{1}^{\lambda}(G)$ by $\|\mu\|(1+\varepsilon)$. By Theorem \ref{montel} (more precisely Corollary~\ref{schonwieder}) there is a subsequence $(N_{k})_{k}$ and $f \in H_{1}^{\lambda}(G)$ such that $p_{u}*\mu*k_{N_{k}}$ converges to $p_{u}*f$ for all $u>0$. This implies that $\widehat{f}(h_{\lambda_{n}})=\widehat{\mu}(h_{\lambda_{n}})$ for all $n$, and so with Lemma \ref{limituto0}
we get that
$$\|f\|_{1}=\lim_{u\to 0} \|p_{u}*f\|_{1} \le (1+\varepsilon)\|\mu\|$$
for all $\varepsilon>0$. This proves the claim.
\end{proof}
We like to mention that Theorem \ref{mainresultinfty}, \ref{brotherRiesz} and \ref{Nabschnitt} have extensions to the setting of vector valued general Dirichlet series (see \cite{Schoolmann}).

\end{document}